\documentclass[twoside,12pt,reqno]{amsart}
\usepackage{bbm, tikz, tikz-cd}
\usetikzlibrary{positioning, arrows}
\usepackage{amssymb}
\usepackage{fullpage}
\usepackage[margin=2cm]{geometry}
\usepackage{bm}
\usepackage{extarrows}

\usepackage{graphicx}
\input{xy}
\xyoption{all}
\xyoption{curve}
\usepackage{color,soul}

\makeatletter

\newtheorem{Theorem}{Theorem}[section]
\newtheorem*{Theorem*}{Theorem}
\newtheorem{Lemma}[Theorem]{Lemma}
\newtheorem{Proposition}[Theorem]{Proposition}
\newtheorem{Corollary}[Theorem]{Corollary}

\theoremstyle{remark}
\newtheorem{Remark}[Theorem]{Remark}
\newtheorem{Example}[Theorem]{Example}

\numberwithin{equation}{section}

\usepackage[pdftex,colorlinks,backref,pagebackref,hypertexnames=false, linkcolor=blue,urlcolor=blue,citecolor=blue]{hyperref}

\newcommand{\arxiv}[1]{{\tt arXiv:#1}}

\let\oldu\u

\def\bar{\overline}

\def\iso{\cong}

\def\onto{\twoheadrightarrow}
\def\isoto{\overset{\sim}{\longrightarrow}}
\DeclareRobustCommand\longtwoheadrightarrow
{\relbar\joinrel\twoheadrightarrow}

\def\epsilon{\varepsilon}

\def\k{{\mathbbm k}}

\newcommand{\Tab}{\operatorname{Tab}}

\def\lmod{\!\operatorname{-mod}}
\def\Lie{\operatorname{Lie}}

\def\cc{{\operatorname{cc}}}

\def\ab{{\operatorname{ab}}}

\def\ad{\operatorname{ad}}

\def\Ann{\operatorname{Ann}}
\def\Ker{\operatorname{Ker}}
\def\col{\operatorname{col}}
\def\gr{\operatorname{gr}}
\def\Mat{\operatorname{Mat}}

\def\Der{\operatorname{Der}}
\def\sspan{\operatorname{span}}
\def\Spec{\operatorname{Spec}}

\def\Quant{\operatorname{Quant}}
\def\Id{\operatorname{Id}}
\def\Ind{\operatorname{Ind}}

\def\End{\operatorname{End}}

\def\Hom{\operatorname{Hom}}
\def\codim{\operatorname{codim}}
\def\Fract{\operatorname{Frac}}
\def\Aut{\operatorname{Aut}}

\def\im{\operatorname{im}}

\def\Prim{\operatorname{Prim}}

\def\Dersh{\D\!\operatorname{er}}

\def\col{\operatorname{col}}
\def\row{\operatorname{row}}
\def\Supp{\operatorname{Supp}}

\def\Irr{\operatorname{Irr}}

\def\bA{{\bm{A}}}

\def\blambda{{\bm{\lambda}}}
\def\bpi{{\bm{\pi}}}

\def\o{\overline}

\def\C{{\mathbb C}}

\def\Z{{\mathbb Z}}
\def\O{{\mathbb O}}
\def\cO{{\overline{\mathbb O}}}

\def\bbA{{\mathbb A}}
\def\bbS{\mathbb{S}}
\def\bbV{\mathbf{V}}

\def\sA{{\mathsf{A}}}
\def\sAa{{\mathsf{A_1}}}
\def\saa{{\mathsf{a_1}}}
\def\sAb{{\mathsf{A_2}}}
\def\sAc{{\mathsf{A_3}}}

\def\sAe{{\mathsf{A_5}}}

\def\sDe{{\mathsf{D_5}}}

\def\sEj{{\mathsf{E_7}}}
\def\sEk{{\mathsf{E_8}}}
\def\sF{{\mathsf{F_4}}}
\def\sG{{\mathsf{G_2}}}

\def\GL{\mathrm{GL}}
\def\SL{\mathrm{SL}}

\def\SO{\mathrm{SO}}
\def\Sp{\mathrm{Sp}}
\def\Tr{\mathrm{Tr}}

\def\b{\mathfrak b}

\def\g{{\mathfrak g}}
\def\gl{\mathfrak{gl}}

\def\l{\mathfrak{l}}

\def\p{\mathfrak p}

\def\t{\mathfrak t}
\def\u{\mathfrak{u}}
\def\v{\mathfrak{v}}

\def\r{\mathfrak{r}}
\def\z{\mathfrak{z}}

\def\b{\mathfrak{b}}

\def\Idl{\mathfrak{Id}}

\def\cA{\mathcal{A}}

\def\U{\mathcal{U}}
\def\E{\mathcal{E}}
\def\B{\mathcal{B}}
\def\D{\mathcal{D}}
\def\Oc{\mathcal{O}}
\def\T{\mathcal{T}}
\def\V{\mathcal{V}}
\newcommand{\A}{\mathcal{A}}
\renewcommand{\S}{\mathcal{S}}
\newcommand{\Nc}{\mathcal{N}}

\renewcommand{\min}{\operatorname{min}}

\title{\boldmath Quantization of nilpotent coadjoint $\GL_N$-orbit closures in positive characteristics}
\author{Filippo Ambrosio, Lewis Topley and Matthew Westaway}
\address{Institut f\"ur Mathematik, Friedrich-Schiller-Universität Jena, Inselplatz 5, 07743 Jena, Germany}
\email{filippo.ambrosio@uni-jena.de}
\address{Department of Mathematical Sciences, University of Bath, Claverton Down, Bath, BA2 7AY}
\email{lt803@bath.ac.uk}
\email{mw2915@bath.ac.uk}


\begin{document}
	
	\maketitle

    \begin{abstract}
    Let $G$ be a reductive group over an algebraically closed field of positive characteristic $p$, good for the root system of $G$. The closures of $G$-orbits in the Hilbert nullcone of the coadjoint representation are conical affine Poisson varieties, generically of full rank, known as {\em nilpotent coadjoint orbits}. In this paper, we classify the filtered Hamiltonian quantizations of these orbit closures for $G = \GL_N$ and any $p > 0$. Our main new technique is a construction of quantizations from certain primitive quotients of the enveloping algebra, inducing them from the stabiliser in $G$ of the Frobenius twisted $p$-character.
    \end{abstract}

 \section{Introduction}

 \subsection{Setting of the paper}
 \label{ss:statement}

 Let $\k$ be an algebraically closed field of positive characteristic $p>0$ and let $G$ be a connected reductive algebraic group such that $(G,G)$ is simply connected, $p$ is good for the root system, and $\g = \Lie G$ admits a non-degenerate $G$-invariant symmetric bilinear form. Under these hypotheses, we refer to $G$ as a {\it standard reductive group} (cf. \cite[\textsection 6.3]{JaLA}). 

  The coadjoint $G$-module $\g^*$ admits the structure of a conic Poisson variety with Hamiltonian $G$-action: the ring of regular functions $\k[\g^*]$ admits a connected grading $\k[\g^*] = \bigoplus_{i \ge 0} \k[\g^*]_i$ with $\g = \k[\g^*]_2$ and a natural Poisson structure satisfying $\{\k[\g^*]_i, \k[\g^*]_j\} \subseteq \k[\g^*]_{i+j-2}$.

 Let $\O\subseteq \g^*$ be a coadjoint $G$-orbit, with Zariski closure $\cO$; this closure is an affine Poisson subvariety of $\g^*$ with generically full rank. The coadjoint $G$-action restricts to a Hamiltonian $G$-action on $\O$, for which the restriction map $\mu^*:\k[\g^*]\twoheadrightarrow \k[\cO]$ is a comoment map. In general, the PBW grading does not descend to $\k[\cO]$; it does, however, precisely when $\O$ is contained in the Hilbert nullcone $\Nc(\g^*)$ of $\g^*$, i.e. is a {\em nilpotent} (coadjoint) orbit. 
 
 For the remainder of the introduction, $\cO$ is the closure of a nilpotent orbit, and therefore a conical Poisson variety with Hamiltonian $G$-action. The main result of this paper is a classification of the set $\Quant^G \k[\cO]$ of Hamiltonian quantizations of $\k[\cO]$ (up to isomorphism) when $G=\GL_N$. That is to say, we classify triples $(\cA,\iota,\Phi)$ where $\cA$ is a filtered $G$-algebra, $\iota:\k[\cO]\xrightarrow{\sim} \gr\cA$ is a $G$-equivariant isomorphism of graded Poisson algebras, and $\Phi:U(\g)\to \cA$ is a $G$-equivariant filtered algebra homomorphism such that $\gr\Phi:\gr U(\g)=\k[\g^*]\to\gr\cA$ satisfies $\gr\Phi=\iota\circ\mu^*$ (see Remark~\ref{R:differentquantizations} for a comparison with formal $\hbar$-quantizations).

Our work follows an expectation expressed in \cite[Remark~5.8]{PT20}, which was inspired by results of Losev classifying quantizations of (closures and affinisations of) nilpotent orbits over $\C$, see \cite{LoQNO, LoOM, LMM}. In particular, we note that the set $\Quant^{\GL_N(\C)} \C[\cO]$ can be identified with the space of characters of a Levi subalgebra of $\gl_N(\C)$, up to the action of a certain finite group (see, for example, \cite[(7.9)]{LMM}).

\subsection{Statement of the results}
\label{ss:statementofresults}

Fix $N > 0$ and let $G = \GL_N$ and $\g = \gl_N$. Set $W \subseteq \GL_N$ to be the group of permutation matrices. These act on $\g$, stabilising the set $\t = \Lie(T)$ of diagonal matrices, and this identifies $W$ with the Weyl group of $(\g, \t)$. 

Fix a partition $\blambda \vdash N$, and let $G_0$ be the Levi subgroup $\GL_{i_1}\times\cdots\times\GL_{i_r}$ of $G$, where $i_j$ denotes the number of parts of $\blambda$ greater than $j-1$. Writing $\g_0=\Lie(G_0)$, the {\em relative Weyl group} of $\g_0$ is $W(\g_0) := N_G(\g_0)/G_0$, a product of symmetric groups permuting the blocks of $\g_0$ of the same size. The composed homomorphism $N_W(\g_0) \to N_G(\g_0) \to W(\g_0)$ is a split surjection, which canonically realises $W(\g_0)$ as a subgroup of $W$. In Section~\ref{S:generallinearalgebra} we realise $W(\g_0)$ combinatorially using a pyramid, and we refer to it as the group of {\it column-swaps}.

Write $\z^*$ for the subspace of $\g_0^*$ consisting of Lie algebra homomorphisms. Note that $\z^*$ is a complete set of non-isomorphic 1-dimensional $\g_0$-modules. By choosing a suitable $\rho\in\t^*$ (see Subsection~\ref{ss:subalgebrasassociated}), the dot-action $w\bullet \zeta=w(\zeta+\rho)-\rho$ of $W$ on $\t^*$ restricts to a dot-action of $W(\g_0)$ on $\z^*$.

The following is the main result of the paper (see Theorems~\ref{T:MainThm} and \ref{C:quantshavecharacters}). Recall that there is a bijection between nilpotent coadjoint $\GL_N$-orbits and partitions of $N$, and write $\O_\blambda$ for the nilpotent orbit corresponding to $\blambda\vdash N$.

\begin{Theorem*}
\begin{enumerate}
    \item There is a natural bijection
    \begin{eqnarray}
    \label{eq:maintheorem}
        \z^* / W(\g_0)_\bullet \xrightarrow{\sim} \Quant^G \k[\cO_\blambda].
    \end{eqnarray}
    \item Quantizations admit central characters and \eqref{eq:maintheorem} respects central characters, in the following sense: if $\A$ is a Hamiltonian quantization of $\k[\o\O_\blambda]$ then the map $\k[\t^*]^{W_\bullet} = U(\g)^G \to \A$ factors through a character of $\k[\t^*]^{W_\bullet}$, giving a map $\Quant^G \k[\o\O_\blambda] \to \t^* / W_\bullet$. The following diagram commutes:
    \begin{eqnarray}
    \label{eq:commdiagCC}
\begin{array}{c}
\begin{tikzpicture}[node distance=1.5cm, auto]
 \node (A) {$\z^*/W(\g_0)_\bullet$};
 \node (B) [right of= A] {};
 \node (C) [right of = B] {$\Quant^G \k[\o\O_\blambda]$};
 \node (D) [below of= C] {$\t^*/W_\bullet$};

 \draw[->] (A) to node {$ $} (C);
 \draw[->] (A) to node {$ $} (D);
 \draw[->] (C) to node {$ $} (D);
\end{tikzpicture}
\end{array}
    \end{eqnarray}
        
\end{enumerate}
\end{Theorem*}

One particular corollary of our result is that the Hamiltonian quantizations of $\Nc(\gl_N^*)$ are precisely the reductions of the universal enveloping algebra by Harish-Chandra central characters. 

Although we need to consider $G=\GL_N$ to get the full result stated above, aspects of our arguments also work outside of type {\sf A}. One application of this is that we can describe a special class of quantizations associated to the minimal nilpotent orbit outside type {\sf A}. In \cite{Jo} Joseph demonstrated that for a complex simple Lie algebra $\g_\C$ of type {\sf B}--{\sf G}, there is a unique (2-sided) ideal $J \subseteq U(\g_\C)$ such that the vanishing locus of $\gr J$ is the minimal nilpotent orbit closure. It is completely prime and maximal. We prove the following analogous result: if $G$ is a standard reductive $\k$-group, with indecomposable root system of type {\sf B}--{\sf G}, then there is a unique ideal such that $\gr J$ is equal to the vanishing ideal of the minimal nilpotent orbit closure. We call this the {\it Joseph ideal}.

\subsection{Sketch of proof and structure of the paper}
\label{ss:sketchofproof}

Maintain the notation used thus far in the introduction. Fix $\chi\in\O_\blambda$, and let $e\in\g$ correspond to $\chi$ under the $G$-equivariant isomorphism $\g \to \g^*$, $X\mapsto (Y\mapsto \Tr(XY))$.

We here briefly sketch the construction of the bijection \eqref{eq:maintheorem} from the main theorem. It is obtained from the following chain of bijections, which we explain over the remainder of this subsection:
$$\z^*/W(\g_0)_\bullet \xleftrightarrow{\tiny\mbox{Thm.}\,\ref{T:OnedimensionalsviaMiura}}\E(\g,e)\xleftrightarrow{\eqref{eq:1dimsandprimitives2}} \Prim_{\chi+X}^{p^{\dim \O_\blambda}} U(\g)\xleftrightarrow{\tiny\mbox{Props.}\,\ref{P:idltoidl},\ref{P:primtoprim}} \Idl_{\o\O_\blambda}^G U(\g)\xleftrightarrow{\tiny\mbox{Prop.}\, \ref{P:quantizationsvsideals}} \Quant^G \k[\overline{\O}_\blambda].$$

The first step is to show that $\Quant^G\k[\cO_\blambda]$ (defined above, and in more detail in  Section~\ref{ss:gradedHamQuant}) is in bijection with the set $\Idl_{\cO_\blambda}^G U(\g)$ of $G$-stable ideals $I \subseteq U(\g)$ such that (under the PBW filtration) $\gr I$ is the defining ideal of $\cO_\blambda$. This is proved in Proposition~\ref{P:quantizationsvsideals}, with the bijection sending $(\cA,\iota,\Phi)$ to $\ker\Phi$.

Recall now that the Lie algebra $\g$ admits a $G$-equivariant restricted structure and write $U_\chi(\g)$ for the reduced enveloping algebra with $p$-character $\chi$ (see Section~\ref{ss:restrictedLiealgebras}). In \cite{PrKW} Premet proved that the dimensions of $U_\chi(\g)$-modules are divisible by $p^{\frac{1}{2}\dim \O_\blambda}$. Following \cite{GT, PT20} the $U_\chi(\g)$-modules of dimension $p^{\frac{1}{2}\dim \O_\blambda}$ will be called {\it small modules}.
Thanks to the Premet--Skryabin equivalence \cite[Theorem~9.2]{GTmodular}, small $U_\chi(\g)$-modules are in bijection with one-dimensional {\em restricted} representations of the {\em finite $W$-algebra} $U(\g,e)$ associated to $\g$ and $e$ (see e.g. \cite[Lemma 2.5]{PT20}). This is an associative algebra which quantizes $\S_\chi:=\chi+\v$, a good transverse slice to $\O_\chi$ at $\chi$ (when working over the complex numbers the Slodowy slice is a common choice of good transverse slice).

The variety of all one-dimensional representations of $U(\g,e)$ is denoted $\E(\g,e)$ and is in bijection with the set of $p^{\frac{1}{2}\dim \O_\blambda}$-dimensional $U_\eta(\g)$-modules for $\eta\in\S_\chi$ (see Subsection~\ref{ss:finWalg}). In fact, modules of such dimension only arise when $\eta$ lies in certain subset of $\S_\chi$ denoted here by $\chi+X$ (this is the {\em Katsylo section} of the unique sheet containing $\O_\blambda$, see Subsection~\ref{ss:sheets}). When these modules do arise they must be simple, and thus (since all simple $U(\g)$-modules are finite-dimensional) they correspond bijectively to a subset of the primitive ideals of $U(\g)$, which we write as $\Prim_{\chi + X}^{p^{\dim \O_\blambda}} U(\g)$. Thus, $\Prim_{\chi + X}^{p^{\dim \O_\blambda}} U(\g)$ is in bijection with $\E(\g,e)$.

Since we assume $G=\GL_N$, there is a natural homomorphism $U(\g,e)\to U(\g_0)$ called the Miura map which induces a map from $\z^*$ (the one-dimensional $U(\g_0)$-modules) to $\E(\g,e)$ (the one-dimensional $U(\g,e)$-modules). In Theorem~\ref{T:OnedimensionalsviaMiura}, we show using the results of \cite{GT} that this map is a quotient mapping for the dot-action of $W(\g_0)$ on $\z^*$, and thus that there is an isomorphism $\z^*/W(\g_0)_\bullet\isoto \E(\g,e)$.

What remains is therefore to show a bijection $\Prim_{\chi+X}^{p^{\dim \O_\blambda}}U(\g)\simeq \Idl_{\o\O_\blambda}^G U(\g)$, and this takes up the majority of Section~\ref{s:Quants}. In Subsection~\ref{ss:quantisationtoprimitiveideals} we construct a map $\Idl_{\o\O_\blambda}^G U(\g)\to \Prim_{\chi+X}^{p^{\dim \O_\blambda}}U(\g)$, in Subsection~\ref{ss:primidlstoquants} we construct a map $\Prim_{\chi+X}^{p^{\dim \O_\blambda}}U(\g) \to \Idl_{\o\O_\blambda}^G U(\g)$, and in Subsection~\ref{ss:classquants} we show that these maps are inverse to each other. To conclude this subsection, we outline the constructions of these two maps; for this, we develop new methods specific to positive characteristics.

We begin with the map $\Idl_{\o\O_\blambda}^G U(\g)\to \Prim_{\chi+X}^{p^{\dim \O_\blambda}}U(\g)$. Recall that $U(\g)$ has a large central subalgebra $Z_p(\g)$, called the {\em $p$-centre} of $U(\g)$, which is isomorphic to $\k[\g^*]^p$. For any $I\in\Idl_{\o\O_\blambda}^G U(\g)$, the {\em $p$-support} of $U(\g)/I$ is the support of $U(\g)/I$ as a $Z_p(\g)$-module, so a subvariety of $\g^*$. In Proposition~\ref{P:psupport} we see that this $p$-support is always the closure of a coadjoint orbit with the same dimension as $\O_\blambda$ (in fact, lying in the same sheet as $\O_\blambda$). Intersecting this orbit with $\S_\chi$ yields a unique element $\eta\in\chi+X$; to see this, we need a new proof of Katsylo's theorem \cite{Kat} which works for $\GL_N$ in positive characteristic, see Theorem~\ref{T:Katsylo}. The ideal of $U(\g)$ generated by $I$ and the maximal ideal of $Z_p(\g)$ corresponding to $\eta$ turns out to lie in $\Prim_{\chi+X}^{p^{\dim \O_\blambda}}U(\g)$, by dimensional considerations (Lemma~\ref{L:IdltoPrim}). This procedure thus gives the required map. 

Let us turn now to the map $\Prim_{\chi+X}^{p^{\dim \O_\blambda}}U(\g)\to \Idl_{\o\O_\blambda}^G U(\g)$. By \cite[Proposition 3.4(1)]{PSk} there exists a $p^{\dim\O_\blambda}$-dimensional Poisson simple algebra $\k_\chi[\g^*]/P_\chi$ such that $\gr(U(\g)/I)\isoto\k_\chi[\g^*]/P_\chi$ for any $I\in \Prim_{\chi+X}^{p^{\dim \O_\blambda}}U(\g)$, so long as we equip $U(\g)/I$ with the Kazhdan filtration corresponding to $\chi$ (see Section~\ref{ss:asscochar} for the definition, and Remark~\ref{R:changinggradings} for further orientation). We reinterpret the coordinate ring $\k[\cO_\blambda]$ using $\k_\chi[\g^*]/P_\chi$, as follows. 

Due to a result of Donkin \cite{Do} $\k[\cO_\blambda]=\k[\O_\blambda]$, and we may identify $\O_\blambda$ with $G/G^\chi$ where $G^\chi$ is the stabiliser of $\chi\in\g^*$ under the coadjoint action. Writing $F:\g^*\to (\g^*)^{(1)}$ for the Frobenius morphism on $\g^*$, we may similarly identify the Frobenius twist $\O^{(1)}$ with $G/G^{F(\chi)}$. Here $G^{F(\chi)}$ is the (non-reduced) scheme-theoretic stabiliser of $F(\chi)$ under the Frobenius-twisted action of $G$ on $(\g^*)^{(1)}$. This group scheme acts on the Poisson algebra $\k_\chi[\g^*]/P_\chi$. We show in Section~\ref{ss:coordringdesc} that $\k[\O_\blambda]$ is isomorphic (as a graded Poisson algebra, with respect to the Kazhdan grading) to the induced algebra $\Ind_{G^{F(\chi)}}^G (\k_\chi[\g^*]/P_\chi)$.

In Section~\ref{ss:primidlstoquants}, we show that the algebra $\cA_I:=\Ind_{G^{F(\chi)}}^G (U(\g)/I)$ can be made into a Hamiltonian quantization $(\cA_I,\iota_I,\Phi_I)$ of $\k[\cO_\blambda]$ for any $I\in \Prim_{\chi+X}^{p^{\dim \O_\blambda}}U(\g)$ (initially with respect to the Kazhdan grading, but through a twisting argument from Subsection~\ref{ss:gradedHamQuant} we may also do so with respect to the PBW grading). Taking $\ker(\Phi_I)$ yields an ideal in $\Idl_{\o\O_\blambda}^G U(\g)$, and this gives us our desired map.

\subsection{Related works}

As noted above, the relationship between quantizations of orbit closures and 1-dimensional representations of the finite $W$-algebra in positive characteristics was suggested by Premet and the second author in \cite[Remark~5.8]{PT20}, and our work builds on some of those ideas. The remarks of {\it loc. cit.} were heavily inspired by Losev: in \cite{LoQNO} the quantizations of equivariant covers of nilpotent orbits over $\C$ were classified. This is closely related to, but distinct from, the problem of classifying quantizations of orbit closures, addressed in this paper. Losev's results show that the quantizations of an orbit cover are classified by those 1-dimensional representations of the finite $W$-algebra which are invariant under the component group of the centraliser of an element in orbit cover. For the general linear group, quantizations are classified by the complex analogue of the variety $\E(\g,e)$ described in Section~\ref{ss:sketchofproof}.

In \cite{LoOM}, Losev classified the quantizations of an arbitrary (complex) conic symplectic singularity, showing that they are parameterised by the quotient of the Namikawa--Cartan space by the Namikawa--Weyl group. Nilpotent $\GL_N(\C)$-orbit closures are examples of such singularities and his parameterising set is analogous to ours.

The study of quantizations of symplectic varieties over fields of positive characteristic was initiated in \cite{BeKa}. When such a variety $X$ is equipped with a {\em restricted structure}, Bezrukavnikov-Kaledin construct a map from the {\em Frobenius constant} quantizations of $X$ to a certain \'{e}tale cohomology group. Here, Frobenius constant quantizations are essentially those for which the Poisson centre of $\Oc_X$ lifts to the centre of the quantization. If the variety $X$ is {\em admissible}, meaning it satisfies a certain cohomology vanishing condition, this map becomes an isomorphism and thus yields a classification of Frobenius constant quantizations. Bezrukavnikov-Kaledin's work was later continued by Bogdanova and Vologodsky \cite{BV} (subsequently joined by Kubrak and Travkin \cite{BKTV}), who correct an inaccuracy in \cite{BeKa} and show a Morita equivalence between a Frobenius constant quantization of a symplectic variety $X$ with restricted structure and a certain central reduction of an algebra of differential operators. Frobenius constant quantizations have since also been studied in \cite{Mu}.

It would be interesting to understand in greater detail the relationship between the work of Bezrukavnikov-Kaledin and Bogdanova-Kubrak-Travkin-Vologodsky, and the results proved in this paper. The Frobenius constant quantizations studied by those authors appear related to Hamiltonian quantizations considered here, specifically those which have nilpotent $p$-support, though it is not clear to us whether the orbits we consider are admissible in the sense of \cite{BeKa}.

\subsection{Layout of the paper}

After this introduction, the paper begins with some general preliminaries (Section~\ref{ss:Prelims}),  followed by discussion of the basic definitions and results from the theories of quantizations (Section~\ref{s:prelimquants}) and finite $W$-algebras (Section~\ref{S:Walgebrassmallmodules}) which are needed for the paper. Section~\ref{s:centralisers} then proves some technical results about coadjoint centralisers and stabilisers which are needed later, ending with a proof of Katsylo's Theorem in positive characteristic for $\GL_N$. Consideration of $G=\GL_N$ continues in Section~\ref{S:generallinearalgebra}, where we show (following \cite{GT}) that in this case one-dimensional representations of the finite $W$-algebras can be classified as quotients of vector spaces by finite groups. Section~\ref{s:Quants} contains the main new results of the paper. After introducing the notion of the $p$-support, we construct bijections between the sets $\Idl_{\o\O_\blambda}^G U(\g)$ and $\Prim_{\chi+X}^{p^{\dim \O_\blambda}}U(\g)$, which, together with the earlier results, yield the main theorem. We conclude the paper in Section~\ref{s:Joseph} by extending a result of Joseph about quantizations of minimal nilpotent orbit closures outside of type $\sA$.

\subsection*{Acknowledgements} The authors would like to thank Ivan Losev, Alexander Premet and Vadim Vologodsky for useful email correspondence, and Simon Riche for some helpful discussions, at the final stages of this project.

The first author's research was carried out at Friedrich-Schiller-Universität Jena. The second and third authors are very grateful for the support provided by UKRI FLF grant numbers MR/S032657/1, MR/S032657/2, MR/S032657/3, as we as UKRI FLF extension grant number MR/Z000394/1.

  \tableofcontents
 
 \section{Preliminaries}\label{ss:Prelims}
 
Throughout, we let $\k = \overline{\k}$ denote an algebraically closed field of characteristic $p > 0$. All unadorned tensor products are taken over $\k$, and all vector spaces, algebras and schemes are defined over $\k$. For a $\k$-algebra $A$, we denote by $A\lmod$ the category of finitely generated $A$-modules.      Furthermore,  all schemes are assumed to be separated and of finite type over $\k$. If $X$ is a scheme, then $\Oc_X$ denotes the structure sheaf and $\T_X = \Dersh \Oc_X$ the tangent sheaf. A rational action of a group scheme is denoted $G \curvearrowright X$.

 \subsection{Filtered vector spaces}
 \label{ss:filteredgradedspaces}
We will use calligraphic fonts to denote filtered vector spaces and Roman fonts for graded vector spaces. Since there can be no confusion we use subscripts to denote both the filtered and graded pieces: $\V = \bigcup_{i} \V_i$ and $V = \bigoplus_i V_i$. All filtrations and gradings are defined over $\Z$. A filtered (resp. graded) linear map $\U\to \V$ (resp. $U\to V$) is required to send $\U_i$ to $\V_i$ (resp. $U_i$ to $V_i$).

A filtration on a vector space $\V$ is always assumed to be ascending, i.e. $\V_i \subseteq \V_{i+1}$ for all $i$. Our filtrations are also assumed to be exhaustive ($\V = \bigcup_i \V_i$) and separated ($\bigcap_i \V_i = 0$), although we do not assume that our filtrations are discrete (i.e. that  $\V_i = 0$ for $i \ll 0$).

The associated graded vector space of $\V$ is $V = \gr \V = \bigoplus_{i\in \Z} V_i$ where $V_i := \V_i / \V_{i-1}$. Since the filtrations are separated, every $v\in \V$ satisfies $v \in \V_d \setminus \V_{d-1}$ for some $d = d(v)$ called the {\it degree of $v$}. Then the {\it symbol} of $v$ is $\bar v := v + \V_{d(v)-1} \in \gr \V$. We say that $\V$ has a {\it filtered basis} if there exists a basis of $\V$ such that the symbols form a basis of $\gr \V$. If the filtration on $\V$ is discrete then an easy inductive argument shows that $\V$ admits a filtered basis.

If $U, V$ are graded vector spaces then the tensor product is graded with $(U\otimes V)_k := \bigoplus_{i + j = k} U_i \otimes V_j$. Similarly if $\U, \V$ are filtered vector spaces then $\U \otimes \V$ is filtered with $(\U\otimes \V)_k := \sum_{i+j = k} \U_i \otimes \V_j$. We define a graded linear map by the rule
\begin{eqnarray}
\label{eq:gradedandtensors}
\begin{array}{ccl}
\gr(\U) \otimes \gr(\V) & \longrightarrow & \gr(\U\otimes \V)\\
(u_i + \U_{i-1}) \otimes (v_j + \V_{j-1}) & \longmapsto &  u_i\otimes v_j + (\U \otimes \V)_{i+j-1}
\end{array}
\end{eqnarray}
In general this map need not be an isomorphism, but it is an isomorphism whenever $\U$ or $\V$ admit filtered bases \cite[Lemma~I.8.2]{NVO}.

Suppose that $V$ is a graded vector space and $\lambda : \k^\times \to \Aut V$ is a rational action preserving the grading. Then we define $V_{i,j}^\lambda := \{ v\in V_i \mid \lambda(t) v = t^jv \text{ for all } t\in \k^\times\}$, and define the {\it $\lambda$-twisted grading} by
\begin{equation}
V^\lambda = \bigoplus_{k \in \Z} V^\lambda_k,\qquad\mbox{where }\,\,\, V^\lambda_k := \bigoplus_{i+j=k} V^\lambda_{i,j}.
\end{equation}
Let $U$ and $V$ be graded vectors spaces equipped with graded $\k^\times$-actions, and abuse notation by letting $\lambda$ denote the $\k^\times$-action on both $U$ and $V$. Then a $\k^\times$-equivariant linear map $U \to V$ is graded with respect to the given gradings if and only if it is graded with respect to the $\lambda$-twisted gradings.

Similarly if $\lambda : \k^\times \to \Aut \V$ is a rational filtered action on a filtered vector space $\V$, then we can define the {\it $\lambda$-twisted filtration} by 
\begin{eqnarray}
\label{eq:twistedfiltration}
\V_k^\lambda := \bigoplus_{i+j=k} \{ v\in \V_i \mid \lambda(t) v = t^jv \text{ for all } t\in \k^\times\}.
\end{eqnarray}
Similar to the graded case, a linear map is filtered $\U \to \V$ if and only if the same map is filtered with respect to the $\lambda$-twisted filtration.

We occasionally write $V = V^\lambda$ (resp. $\V = \V^\lambda$) to emphasise that we are considering the $\lambda$-twisted grading (resp. $\lambda$-twisted filtration).

\begin{Lemma}
\label{L:twistingfiltrations}
$\gr(\V^\lambda) = \gr(\V)^\lambda$ as graded vector spaces. If $V$ is an algebra and $\lambda: \k^\times \to \Aut(V)$ acts by algebra automorphisms then $V^\lambda = \bigoplus_k V^\lambda_k$ is an algebra grading. $\hfill \qed$ 
\end{Lemma}

\subsection{Frobenius twists}

Given a vector space $V$, we denote by $V^{(1)}$ the vector space with the same underlying abelian group as $V$ but with scalar multiplication given by $(\lambda,v)\mapsto \lambda^{1/p} v$. This gives a monoidal endofunctor on the category of vector spaces, and so if $A$ is a $\k$-algebra then $A^{(1)}$ is also a $\k$-algebra.

For a scheme $X$ with regular functions $\Oc_X$, the {\em Frobenius twist} of $X$ is the scheme $X^{(1)}$ which is equal to $X$ as a set, but with sheaf of functions $\Oc_X^{(1)}$.
When $X$ is reduced the map $\Oc_X^{(1)}\to \Oc_X^p$, $f\mapsto f^p$, is an isomorphism of $\k$-sheaves on $X$, and so we can regard $X^{(1)}$ as the scheme $(X, \Oc_X^p)$ for simplicity. 
This is especially transparent for reduced affine schemes, where $X^{(1)}$ is isomorphic to $\Spec (\k[X]^p)$. 
The morphism $F_X: X \to X^{(1)}$ corresponding to $\Oc_X^{(1)}\cong\Oc_X^p\hookrightarrow\Oc_X$ is called the {\em Frobenius morphism} for $X$.

If $G$ is a group scheme then $F_G : G \to G^{(1)}$ is a homomorphism of group schemes, and the kernel $G_1 := \Ker F_G$ is called the {\it (first) Frobenius kernel} of $G$. The latter is an infinitesimal group scheme: it has a single $\k$-point but a nontrivial sheaf of sections with the structure of a (finite-dimensional) Hopf algebra.

\subsection{Group schemes and coadjoint orbits}
\label{ss:groupschemesandorbits}

	Let $G$ be a reductive algebraic group over $\k$, and let $\g=\Lie(G)$. We assume that $G$ satisfies the standard hypotheses from \cite[\textsection 6.3]{JaLA}, that is (A) that the derived subgroup of $G$ is simply connected, (B) that $p$ is good for $G$ and (C) that there exists a $G$-invariant non-degenerate symmetric bilinear form on $\g$. The latter implies in particular that there exists an isomorphism of $G$-modules
	\begin{eqnarray}
    \label{eq:kappa}
	\kappa:\g\xrightarrow{\sim}\g^{*}.
	\end{eqnarray}
    We refer to a group under these hypotheses as a {\it standard reductive group}.
    
		Since $G$ acts on $\g^{*}$, we may study the orbits of this action; through the isomorphism $\kappa$ this is equivalent to studying the orbits of $G$ on $\g$ under the adjoint action. Given $\chi\in\g^{*}$, the $G$-orbit $\O_\chi:=G\cdot\chi$ is a smooth quasi-affine subvariety of $\g^{*}$, and we denote by $\cO_\chi$ its (Zariski) closure in $\g^{*}$. Such orbits are all even-dimensional \cite[\textsection 2.8]{JaNO}. 
		
		We write $G^\chi$ for the coadjoint stabiliser, which will always be taken scheme-theoretically. For standard reductive groups, these stabilisers are smooth (i.e. reduced) group schemes \cite[\textsection 2.9]{JaNO} and so 
	\begin{eqnarray}
	\label{L:smoothstabilisers}
	\g^\chi = \Lie G^\chi \text{ for all } \chi \in \g^*.
	\end{eqnarray}
	Equivalently \cite[\textsection 2.2]{JaNO} there is an isomorphism of varieties 
\begin{eqnarray}
\label{eq:quotientisorbit}
G/G^\chi\xrightarrow{\sim}\O_\chi \text{ for all } \chi \in \g^*.
\end{eqnarray}	
Identifying through this isomorphism we have a sequence of natural inclusions
\begin{eqnarray}
\label{eq:naturalinclusions}
    \k[\cO_\chi]\subseteq \k[\O_\chi]= \k[G/G^\chi] = \k[G]^{G_\chi} \subseteq \k[G].
\end{eqnarray}
By \cite[Corollary 8.3]{JaNO} the first inclusion is an equality if and only if $\cO_\chi$ is normal.
\begin{Lemma}
\cite{Do}
\label{L:normalorbits}
For $G = \GL_n$, coadjoint orbit closures are normal, and so
$\k[\O_\chi] =  \k[\overline{\O}_\chi]$ for each such orbit.
\end{Lemma}

\subsection{The twisted coadjoint representation}
\label{ss:twistedcoadjoint}

Retain the assumptions of Section ~\ref{ss:groupschemesandorbits}. Then $G$ acts on $(\g^*)^{(1)}$ by the composition
$$G \overset{F_G}{\longrightarrow} G^{(1)} \longrightarrow \GL((\g^*)^{(1)}).$$ The Frobenius morphism $F = F_{\g^*} : \g^* \to (\g^*)^{(1)}$ is a $G$-equivariant bijective morphism of schemes, however the inverse is not a morphism. Consequently, $F$ induces a bijection on $G$-orbit spaces, but does not respect stabilisers. If $\chi \in \g^*$ then we write $G^{F(\chi)}$ for the stabiliser of $F(\chi)$ in $G$, and we note that $G^\chi \subsetneq G^{F(\chi)}$ in general. Directly from the definitions we see $G_1, G^\chi \subseteq G^{F(\chi)}$ as closed subgroup schemes.

By definition, $G^{F(\chi)} = F_G^{-1}\big((G^\chi)^{(1)}\big)$, which implies
\begin{eqnarray}
\label{eq:frobstabiliserfibreproduct}
    G^{F(\chi)} \cong G \times_{G^{(1)}} (G^\chi)^{(1)}. 
\end{eqnarray}
In the language of \cite{BMR}, $G^{F(\chi)}$ is a Frobenius neighborhood of $G^\chi$ inside $G$. On the level of coordinate rings, \eqref{eq:frobstabiliserfibreproduct} becomes 
\begin{eqnarray}
\label{eq:coordfrobstabiliserfibreproduct}
    \k[G^{F(\chi)}] \cong \k[G] \otimes_{\k[G]^p} \k[G^\chi]^p. 
\end{eqnarray}

Identifying $G^{F(\chi)}$ with the fibre product through \eqref{eq:frobstabiliserfibreproduct}, the subgroup $G^\chi\subseteq G^{F(\chi)}$ identifies with
\begin{eqnarray}
    G^\chi = G^\chi \times_{G^{(1)}} (G^\chi)^{(1)}.
\end{eqnarray}
Since $G \times_{G^{(1)}} (G^\chi)^{(1)}$ is generated by $G^\chi \times_{G^{(1)}} (G^\chi)^{(1)}$ and $G_1 = G \times_{G^{(1)}} \{1_{(G^\chi)^{(1)}}\}$, it follows that $G^{F(\chi)}$ is generated by the subschemes $G_1$ and $G^\chi$.

By \cite[I.5.6(2)]{JanRAGS} there is an isomorphism 
\begin{eqnarray}
    G / G^{F(\chi)} \isoto \O_{\chi}^{(1)}.
\end{eqnarray}
	
\subsection{Restricted Lie algebras} \label{ss:restrictedLiealgebras}
The PBW grading $\k[\g^*] = \bigoplus_{i \ge 0} \k[\g^*]_i$ is defined by placing $\g$ in degree 2. We adopt a similar convention for the PBW filtration on the enveloping algebra $U(\g)$, placing $\g$ in degree 2. By the PBW theorem we can (and shall) identify $\gr U(\g) = \k[\g^*]$ as graded algebras.

The Lie algebra $\g$ admits a canonical choice of $G$-invariant restricted structure $x \mapsto x^{[p]}$, taking $\delta\in\Der \k[G]$ to $\delta^p\in\Der \k[G]$ (see \cite[\textsection 2]{SF80} for a general reference on restricted structures). This gives rise to the notion of semisimple and nilpotent elements of $\g$. The {\em nilpotent cone} is the reduced scheme
$$\Nc(\g) = \{x \in \g \mid x^{[p]^{\dim \g}} = 0\}.$$
For $G = \GL_N$ the closed points of $\Nc(\g)$ are those which act nilpotently on $\k^N$. 
We write $\Nc(\g^*) = \kappa (\Nc(\g) ) \subseteq \g^*$, and note that $\Nc(\g^*)$ coincides with the Hilbert nullcone of the coadjoint representation, i.e. the set of $\chi \in \g^*$ such that $0 \in \overline{G\cdot \chi}$ \cite{PrKR}.

The map $\xi:\g\to U(\g)$ given by $x\mapsto x^p - x^{[p]}$ is $p$-semilinear and takes values in the centre $Z(\g)$ of $U(\g)$. The image of this map generates a subalgebra $Z_p(\g) \subseteq U(\g)$, called the {\it $p$-centre} of $U(\g)$ (also known as the {\it Frobenius centre}).

The PBW theorem implies two important properties of $Z_p(\g)$: (i) $U(\g)$ is a free $Z_p(\g)$-module of rank $p^{\dim \g}$; and (ii) $\xi$ induces a $G$-equivariant isomorphism $\k[(\g^*)^{(1)}] = S(\g)^{(1)} \to Z_p(\g)$. Throughout this paper we make the identification $(\g^*)^{(1)} = \Spec Z_p(\g)$. This parameterisation of maximal ideals can be made explicit: for $\chi \in \g^*$ we have the maximal ideal 
\begin{eqnarray}
    \label{eq:idealIchi}
    I_\chi := (x^p - x^{[p]} -\chi(x)^p \mid x\in \g) \unlhd Z_p(\g).
\end{eqnarray}

Every simple $\g$-module admits a $p$-central character, and the $\g$-modules with $p$-central character $\chi$ are precisely the $U_\chi(\g)$-modules, where  $U_\chi(\g) := U(\g) / U(\g) I_\chi$ is the {\em $\chi$-reduced enveloping algebra}. These algebras form a flat family over $(\g^*)^{(1)}$, with fibres $U_\chi(\g)$ all of dimension $p^{\dim \g}$.

The reduced enveloping algebras admit semiclassical analogues: $\k[\g^*]$ is a Poisson algebra and $\k[\g^*]^p = \k[(\g^*)^{(1)}]$ is contained in the Poisson centre. The {\it $\chi$-reduction of} $\k[\g^*]$ is $\k_\chi[\g^*] := \k[\g^*] / J_\chi$, where
\begin{eqnarray}
\label{eq:idealJchi}
J_\chi = (x^p - \chi(x)^p \mid x\in \g)\unlhd \k[\g^*]^p.    
\end{eqnarray}
The adjoint action of $\g$ on $\k[\g^*]$ descends to an action on $\k_\chi[\g^*]$. We remark that these algebras appeared in various works of Premet (see \cite{PSk} and references therein), who referred to them as reduced symmetric algebras and denoted them $S_\chi(\g)$.

It is not hard to see that $\k_\chi[\g^*]$ admits a unique maximal $\g$-stable ideal $P_\chi$, which coincides with the sum of all proper $\g$-stable ideals (this is proper since $\k_\chi[\g^*]$ has a unique maximal ideal). The Hamiltonian derivations of $\k[\g^*]$ coincide with $\k[\g^*] \ad(\g)$ and so it follows that an ideal of $\k_\chi[\g^*]$ is $\g$-stable if and only if it is a Poisson ideal. It follows that $\k_\chi[\g^*]$ admits a unique maximal Poisson ideal. Combining \cite[Proposition~3.4]{PSk} and \eqref{L:smoothstabilisers}, we deduce the ideal has codimension $p^{\dim \O_\chi}$. We record these facts in the form of a lemma.

\begin{Lemma}
\label{L:PremetSkryabin}
$\k_\chi[\g^*]$ admits a unique maximal Poisson ideal $P_\chi$, which has codimension $p^{\dim \O_\chi}$.
\end{Lemma}

In \cite{PrKW} Premet proved the second Kac--Weisfeiler conjecture, showing that
\begin{eqnarray}
\label{eq:Premetstheorem}
p^{\dim \O_\chi} \text{ divides } \dim(V)^2 \text{ for } V \in U_\chi(\g)\lmod.
\end{eqnarray}
We say that $V\in U_\chi(\g)\lmod$ is a {\it small module} if $\dim(V)^2 = p^{\dim \O_\chi}$.

The result \eqref{eq:Premetstheorem} was strengthened by the proof of a conjecture by Humphreys in \cite{PT20}. The authors showed that $U_\chi(\g)$ has a simple module of dimension $p^{\frac{1}{2}\dim \O_\chi}$ and so the dimension bound implied by \eqref{eq:Premetstheorem} is sharp. We record a version of this observation formally, for later use.
\begin{Lemma}
\label{L:KW+Humphreys}
For $d \ge 0$, the following are equivalent:
\begin{enumerate}
    \item $U_\chi(\g)$ admits a representation of dimension $p^d$;
    \item $\frac{1}{2}\dim \O_\chi \le d$.
\end{enumerate}
\end{Lemma}

\section{Quantizations}\label{s:prelimquants}
In this section, we recall the notion of quantizations and Hamiltionian quantizations we use throughout the paper. In this and subsequent sections, our schemes will always be quasi-affine, i.e. they  embed into their affinization $X \hookrightarrow \Spec(\Gamma(X,\Oc_X))$. 

\subsection{Poisson schemes}

We say that $X$ is a {\em Poisson scheme} if $\Oc_X$ is equipped with a Lie bracket
$$\{ \cdot , \cdot \} : \Oc_X \times \Oc_X \longrightarrow \Oc_X$$
which is a biderivation.

If $X$ is affine, then specifying a rational action $\k^\times \curvearrowright X$ is equivalent to specifying an action $\k^\times \curvearrowright \k[X]= \Gamma(X,\Oc_X)$ by automorphisms, which is equivalent to specifying a grading $\k[X] = \bigoplus_{i\in \Z} \k[X]_i$. For $d > 0$ we say that the Poisson bracket {\it lies in degree} $-d$ if $\{\k[X]_i, \k[X]_j\} \subseteq \k[X]_{i+j-d}.$

The {\it Hamiltonian derivation} associated with a local section $f\in \Oc_X(U)$ is $\{f, -\} \in \T_X(U)$. Under our standing assumptions $X$ is regular precisely when it is a smooth affine variety. In this case, the Poisson structure on $X$ is determined by a global section $\pi \in \Gamma(X, \bigwedge^2 \T_X)$ and the Hamiltonian derivations form a $\Oc_X$-submodule of $\T_X$. We say that $X$ is {\it symplectic} if every section of $\T_X$ is Hamiltonian.

\begin{Example}
\label{ex:symplecticaffinespace}
Let $X = \bbA^{2n}$ and equip $\k[X] = \k[p_{i}, q_i \mid i=1,\ldots,n]$ with a Poisson bracket via $\{p_i, q_j\} = \delta_{i,j}$, $\{p_i, p_j\} = \{q_i, q_j\} = 0$. This Poisson structure is the symplectic affine space, and has Poisson bracket in degree $-2$.
\end{Example}

\begin{Example}
\label{ex:PoissonLiestructure}
If $\g$ is a finite-dimensional Lie algebra then $\g^*$ is a Poisson $\k^\times$-variety with bracket in degree $-2$. Here the $\k^\times$-action is restricted from the vector space structure on $\g^*$.
\end{Example}

\subsection{Graded Hamiltonian actions}
\label{ss:gradedHamactions}

Let $X$ be an affine Poisson scheme, and let $G$ be a standard reductive group acting rationally on $X$ by Poisson automorphisms. Write $\alpha : G \to \Aut\k[X]$ for the action induced by $G\curvearrowright X$. We say that the action is {\it Hamiltonian} if there exists a $G$-equivariant {\it comoment map} $\mu^* : \k[\g^*] \to \k[X]$ such that for $x\in \g$ we have $d_e\alpha(x) = \{\mu^*(x), - \} \in \Der \k[X]$.

Equip $\k[\g^*]$ with the (doubled) PBW grading $\bigoplus_{i\in \Z} \k[\g^*]_i$. We fix a homomorphism $\lambda : \k^\times \to G$ and recall the $\lambda$-twisted PBW grading $\k[\g^*]^\lambda = \bigoplus_{i\in \Z} \k[\g^*]^\lambda_i$ (cf. Subsection~\ref{ss:filteredgradedspaces}; if $\lambda$ is trivial we omit the superscript). Explicitly, $\k[\g^*]_i^\lambda$ is spanned by $\lambda(t)$-eigenvectors of eigenvalue $t^j$ of degree $d$ such that $d + j = i$. Note that the Poisson bracket lies in degree $-2$.

The action of $G$ on itself by conjugation gives rise to an action $G \curvearrowright \k[G]$, and we equip $\k[G]$ with the $\lambda$-eigenspace grading, explicitly $\k[G]_i^\lambda = \{f\in \k[G] \mid \lambda(t) \cdot f = t^i f \text{ for } t\in \k^\times\}$. The coaction map associated to $G \curvearrowright \k[\g^*]$ is graded
$$\k[\g^*]_i^\lambda \to \bigoplus_{j+k=i} \k[\g^*]_j^\lambda \otimes \k[G]_k^\lambda$$
and we say that the $G$-action on $\k[\g^*]$ respects the grading. We remark that $G$ does not preserve the graded pieces. 

We say that a Hamiltonian action $G \curvearrowright X$ is {\em $\lambda$-graded} if $\k[X]$ is graded with bracket in degree $-2$, the coaction map for $G\curvearrowright\k[X]$ is graded, and the comoment map is graded (with respect to the $\lambda$-twisted grading).

The fundamental example of a $\lambda$-graded Hamiltonian action, for our purposes, is the coadjoint representation $G\curvearrowright \g^*$ where $\g = \Lie(G)$. The comoment map is the identity on $\k[\g^*]$. A closed $G$-stable subvariety $X \subseteq \g^*$ admits a $\lambda$-graded Hamiltonian action if and only if $X$ is $\k^\times$-stable for the $\k^\times$-action coming from the vector space structure on $\g^*$ ($\k^\times$-stable will always be understood with respect to the vector space $\k^\times$-action).

Now let $X\subseteq \g^*$ be a $G$-stable affine subscheme. The defining ideal $I_X \unlhd \k[\g^*]$ is $\ad(\g)$-stable, hence stable under all Hamiltonian derivations. It follows that such a scheme is Poisson, and moreover that the action $G\curvearrowright\k[X]$ is Hamiltonian with comoment map $\k[\g^*] \onto \k[X]$ the natural projection i.e. restriction of regular functions. 

\begin{Lemma}\label{L:gradorbs}
Let $G$ be a standard reductive group, and $\chi \in \g^*$. The Hamiltonian action $G\curvearrowright \overline{\O}_\chi$ is $\lambda$-graded if and only if $\chi \in \Nc(\g^*)$.
\end{Lemma}
\begin{proof}
We show that $\O_\chi$ is $\k^\times$-stable if and only if $\chi \in \Nc(\g^*)$. This will imply the statement of the  current lemma. The `if' implication is \cite[Lemma~2.10]{JaNO}.

The central fibre of the adjoint quotient map $\pi: \g^* \to \g^*/\!\!/ G$ is precisely $\Nc(\g^*)$ \cite[\textsection\textsection 6--7]{JaNO}. Since $\k[\g^*]^G$ is a PBW graded subalgebra of $\k[\g^*]$, there is a natural $\k^\times$-action on $\g^*/\!\!/ G$ for which $\pi$ is $\k^\times$-equivariant, with unique fixed point $\pi(0)$. We denote this action by $z \cdot x$ for $x \in \g^*/\!/G$ and $z\in \k^\times$.

It follows that if $\chi\in \g^* \setminus \pi^{-1}(0)$ then since $\pi(\chi)$ is not $\k^\times$-fixed there exists $z\in \k^\times$ such that $\pi(\chi) \neq z\cdot \pi(\chi) = \pi(z\chi)$. Since $z\chi$ and $\chi$ lie in distinct fibres of $\pi$, they lie in distinct orbits.
\end{proof}

\subsection{Graded Hamiltonian quantizations}
\label{ss:gradedHamQuant}

Let $\A$ be a filtered algebra $\A = \bigcup_{i\in \Z} \A_i$. The associated graded algebra is $\gr \A = \bigoplus_{i\in \Z} \A_i / \A_{i-1}$. We say that $\A$ is {\it almost commutative of degree $-2$} if $[\A_i, \A_j] \subseteq \A_{i + j - 2}$. In this case, $\gr \A$ is a commutative algebra equipped with a Poisson bracket in degree $-2$ via
$$\{ a + \A_{i-1}, b + \A_{j-1}\} = [a,b] + \A_{i+j-3}$$
for $a\in \A_{i}$ and $b\in \A_j$.

If $A$ is a graded Poisson algebra with Poisson bracket in degree $-2$, then a {\it filtered quantization} of $A$ is a pair $(\A, \iota)$ of a filtered algebra $\A$, almost commutative of degree $-2$, and a graded Poisson isomorphism $\iota : A \xrightarrow{\sim} \gr \A$. An isomorphism of quantizations $(\A, \iota)$ and $(\A', \iota')$ is an isomorphism of filtered algebras $\phi : \A \to \A'$ that induces a commuting diagram of Poisson isomorphisms:
\begin{center}
\label{eq:quantisodiagram}
\begin{tikzcd}
 &  A\arrow[dr, "\iota"] \arrow[swap, dl, "\iota'"] & \\
\gr \A \arrow[rr, "\phi"] & & \gr \A'.
\end{tikzcd}
\end{center}

\begin{Remark}
\label{R:differentquantizations}
If $A$ is a graded Poisson algebra with bracket in degree $-2$, there are three common notions of a quantization of $A$, all of which are equivalent.

Instead of filtered quantizations, one may consider flat graded $\k[\hbar]$-algebras $\A_\hbar$ with $\deg(\hbar) = 1$, $[\A_\hbar, \A_\hbar] \subseteq \hbar^2 \A_\hbar$ with a fixed Poisson isomorphism $\iota : \A_\hbar / (\hbar) \to A$. These are in bijection with filtered quantizations via the Rees algebra construction, and specialisation $\A_\hbar \mapsto \A_\hbar / (\hbar - 1)$.

One may also consider a topologically graded $\hbar$-adically complete, flat $\k[[\hbar]]$-algebra $\A_\hbar^\vee$ such that $[\A_\hbar^\vee, \A_\hbar^\vee] \subseteq \hbar^2 \A_\hbar^\vee$ equipped with a Poisson isomorphism $\A_\hbar^\vee / (\hbar) \to A$ as above. Algebras satisfying these properties are in bijection with those of the previous paragraph, via $\hbar$-adic completion. To invert this procedure, one should take the $\k^\times$-locally finite part of $\A_\hbar^\vee$.
\end{Remark}

Suppose now that $X$ is a $\k^\times$-stable affine Poisson variety with bracket in degree $-2$ and $A = \k[X]$. Fix a cocharacter $\lambda : \k^\times \to G$, and suppose that $G \curvearrowright X$ is a $\lambda$-graded Hamiltonian action. Equip $U(\g)$ with the $\lambda$-twisted PBW filtration $U(\g)^\lambda$, which places $x\in \g$ in degree $i + 2$ whenever $\lambda(t) x = t^i x$ for all $t \in \k^\times$. Note that $U(\g)^\lambda$ is almost commutative in degree $-2$. Using the PBW theorem, and Lemma~\ref{L:twistingfiltrations}, we identify $\gr U(\g)^\lambda = \k[\g^*]^\lambda$ as graded Poisson algebras.

We have equipped $\k[G]$ with the grading by $\lambda(\k^\times)$-eigenspaces, and we write $\k[G]_{\le i}^\lambda$ for the pieces of the induced filtration. A quantization $\A$ of $\k[X]$ is called {\it $\lambda$-Hamiltonian} if the following two conditions hold:\vspace{4pt}
\begin{enumerate}
\setlength{\itemsep}{4pt}
\item[(Q1)]  There is a rational action $G\curvearrowright \A$ such that the coaction map is filtered, i.e. 
$$\A_i \to \sum_{j+k = i} \A_j \otimes \k[G]_{\le k}^\lambda,$$
and the corresponding associated graded map is the coaction map for $G \curvearrowright X$. Note that $\gr (\A \otimes \k[G]^\lambda)$ identifies with $\k[X] \otimes \k[G]^\lambda$ since $\k[G]^\lambda$ has a filtered basis.
\item[(Q2)] There exists a $G$-equivariant filtered homomorphism $\Phi : U(\g)^\lambda \to \A$  such that $\gr \Phi = \iota \circ \mu^*$.\vspace{4pt}
\end{enumerate}

We usually denote a Hamiltonian quantization by $(\A, \iota, \Phi)$, and we call $\Phi$ the {\it quantum comoment map}. An isomorphism of Hamiltonian quantizations $(\cA,\iota,\Phi)\to (\cA',\iota',\Phi')$ is an isomorphism of quantizations $\phi:(\cA,\iota)\to(\cA',\iota')$ such that $\phi\circ\Phi=\Phi'$.

We write $\Quant^G(A)$ for the set of isomorphism classes of Hamiltonian quantizations of $A := \k[X]$. We emphasise that a fixed choice of cocharacter $\lambda \in X_*(G)$ is embedded in the definition of a Hamiltonian quantization, and we sometimes write $\Quant^G(A^\lambda)$ to emphasise this choice. The following lemma is straightforward.
\begin{Lemma}\label{L:twistquant}
    Fix a cocharacter $\lambda:\k^\times\to G$. There is a natural bijection
    \begin{eqnarray*}
    \begin{array}{rcl}
        \Quant^G(A) & \xrightarrow{\,\,\,\sim\,\,\,} & \Quant^G(A^\lambda)\\
        \A & \longmapsto & \A^\lambda.
        \end{array}
    \end{eqnarray*}
\end{Lemma}
\begin{Remark}
\label{R:changinggradings}
    Throughout Section~\ref{s:prelimquants} we have used the notation $A^\lambda$ from Section~\ref{ss:filteredgradedspaces} to denote a twist of the PBW grading. In later sections, we occasionally suppress the superscript for simplicity, but this should cause no confusion since all such abuses are clearly announced. We advise the reader to pay special attention to the grading used in Subsections~\ref{ss:Katsylosection}, \ref{ss:primidlstoquants} and \ref{ss:classquants}. 
\end{Remark}

\subsection{Quantizations of closed subschemes of the coadjoint representation}
\label{ss:quantclosedsubschemes}

Let $G$ be a standard reductive algebraic group, $\g = \Lie G$ and fix a cocharacter $\lambda :  \k^\times  \to G$. Let $X \subseteq \g^*$ be a closed $G\times \k^\times$-stable subscheme and write $I_X$ for the defining ideal of $X$. In this case, $X$ is Poisson, the action $G \curvearrowright X$ is graded Hamiltonian, and the comoment map is surjective.

If $(\A, \iota, \Phi)$ is a Hamiltonian quantization of $\k[X]$, then $\Phi$ is surjective; indeed, by Lemma~\ref{L:twistingfiltrations} we may reduce to the case where $\lambda$ is trivial, in which case the filtration is discrete and we may apply \cite[Corollary 6.14]{MR}. A similar argument shows that $\Phi$ is strict (in the language of \cite{MR}), which yields that $\gr \ker \Phi = \ker \mu^* = I_X$ using (Q2). Write $\Idl_X^G U(\g)^\lambda$ for the set of $G$-stable ideals $J \unlhd U(\g)$ such that $\gr J = I_X$.

\begin{Proposition}
\label{P:quantizationsvsideals}
The following map is a bijection
    \begin{eqnarray}
    \label{eq:quantstoideals}
        \begin{array}{rcl}
        \Quant^G \k[X]^\lambda & \longrightarrow & \Idl_X^G U(\g)^\lambda\\
        (\A, \iota, \Phi) & \longmapsto & \ker \Phi.
        \end{array}
    \end{eqnarray}
\end{Proposition}
\begin{proof}
As noted above, \eqref{eq:quantstoideals} takes values in the correct set. Let $(\A, \iota, \Phi)$, $(\A', \iota', \Phi')\in\Quant^G  \k[X]^\lambda$ such that there exists an isomorphism of Hamiltonian quantizations $\Psi : \A \to \A'$. By definition $\Psi\circ\Phi=\Phi'$ and so $\ker\Phi=\ker\Phi'$. Hence \eqref{eq:quantstoideals} is well-defined.

We now construct an inverse to \eqref{eq:quantstoideals}. For $J\in \Idl_X^G U(\g)^\lambda$, let $\Phi_J:U(\g)\to U(\g)/J$ denote the natural projection. Then $((U(\g)/J)^\lambda,\Id_{\k[X]},\Phi_J)\in\Quant^G \k[X]$. This is clearly a right inverse to \eqref{eq:quantstoideals}; we show it is a left inverse.

Let $(\A, \iota, \Phi)\in\Quant^G  \k[X]^\lambda$ and write $J = \Ker \Phi$.  
Since $\Phi$ factors through $\Phi_J$ we obtain an isomorphism of filtered algebras $\bar\Phi : (U(\g) / J)^\lambda \to \A$ satisfying $ \overline{\Phi} \circ \Phi_J=\Phi$.
Combining the equalities $\gr\Phi=\iota\circ\mu^*$ (from (Q2)) and $\gr\Phi_J=\Id_{\k[X]}\circ \mu^*$ with the surjectivity of $\mu^*$ yields $\gr\o\Phi\circ\Id_{\k[X]}=\iota$. Thus $\bar\Phi$ determines an isomorphism of quantizations $((U(\g) / J)^\lambda, \Id_{\k[X]}, \Phi_J) \to (\A, \iota, \Phi)$, and the result follows.
\end{proof}

 \section{Finite $W$-algebras}
 \label{S:Walgebrassmallmodules}
Throughout this section, $G$ is a standard reductive group over an algebraically closed field $\k$ of characteristic $p>0$. In particular, there exists a non-degenerate $G$-invariant symmetric bilinear form $\o{\kappa}:\g\times\g\to\k$, and thus a $G$-equivariant isomorphism $\kappa \colon \g\xrightarrow{\sim} \g^*$, $x\mapsto \o\kappa(x,-)$ as in \eqref{eq:kappa}.

\subsection{Associated cocharacters and the Kazhdan grading}
\label{ss:asscochar}

Fix $e\in \Nc(\g)$ and let $\chi = \kappa(e) \in \g^*$. Define $\lambda = \lambda_e \colon \k^\times \to G$ to be a representative of the $(G^e)^\circ$-conjugacy class of cocharacters {\it associated to $e$}, in the sense of \cite[\textsection 5.3]{JaNO} (see also \cite{PrKR} where the same class of cocharacters were characterised via Kempf--Rousseau theory, and named {\it optimal cocharacters}).

The grading $\g = \bigoplus_{i\in \Z} \g(i)$ given by $\g(i) := \{ x\in \g \mid \lambda(t)x = t^i x \text{ for } t\in \k^\times\}$ is known as the {\it Dynkin grading}. It satisfies two key properties:
\begin{enumerate}
\item $e \in \g(2)$;
\item $\g^e \subseteq \g(\ge 0)$.
\end{enumerate}

We occasionally use the notation $\g(\ge \! i) = \bigoplus_{j\ge i} \g(j)$, and similar for $\g(\le i)$ and $\g(<\!i)$ or $\g(>\!i)$. We note that $\g(\ge \! 0)$ and $\g(\le \!0)$ are parabolic subalgebras, and we denote the unipotent radicals of their parabolic subgroups by $G(< \!0)$ and $G(>\!0)$.

The {\it Kazhdan grading (associated to $\lambda_e$)} on $\k[\g^*]$ is defined by placing the generators $\g(i)$ in degree $i+2$, and the Kazhdan filtration on $U(\g)$ is defined similarly. This coincides with the $\lambda_e$-twisted grading on $\k[\g^*]$ and $\lambda_e$-twisted filtration on $U(\g)$, in the notation of Section~\ref{ss:filteredgradedspaces}.

We define the Kazhdan $\k^\times$-action $\gamma_e : \k^\times\to \GL(\g)$ by 
\begin{eqnarray}
\label{eq:Kazhdanaction}
\gamma_e(t) x = t^2 \lambda_e(t) x \text{ for } x\in \g.    
\end{eqnarray}

We extend the action of $\gamma_e(\k^\times)$ on $\g$ to $\k[\g^*]$, so the Kazhdan grading coincides with the grading by characters of $\gamma_e(\k^\times)$. A reduced affine subscheme $X\subseteq \g^*$ is $\gamma_e(\k^\times)$-stable if and only if $\k[\g^*] \to \k[X]$ is a graded homomorphism.

Recall the notation \eqref{eq:idealIchi} and \eqref{eq:idealJchi}. The ideal $J_\chi \unlhd \k[\g^*]$ is a Kazhdan graded ideal, and so the $\chi$-reduced algebra
$\k_\chi[\g^*]$ is Kazhdan graded. Similarly the Kazhdan filtration on $U(\g)$ descends to the reduced enveloping algebra $U_\chi(\g)$.

\begin{Lemma}\label{L:grUchi}
    For $\eta \in \chi + \kappa \g(\le 1)$ we have a natural isomorphism of Kazhdan graded algebras $$\gr U_\eta(\g) \cong \k_\chi[\g^*].$$
\end{Lemma}
\begin{proof}
        Since $\dim U_\eta(\g) = p^{\dim \g} = \dim \k_\chi[\g^*]$ and $\k[\g^*] \onto \gr U_\eta(\g)$ it suffices to show that $J_\chi \subseteq \gr I_\eta$.
        
        Since $e\in \g(2)$ we see that homogeneous $x \in \g(i)$ satisfies $\chi(x) = 0$ for $i \ne -2$. 
        Similarly, if $\eta \in \chi + \kappa\g(\le 1)$ then $\eta(x) = \chi(x)$ for $x \in \g(i), i\le -2$.

        If $x\in \g(i)$ then $x^p$ lies in Kazhdan degree $ip + 2p$ whilst $x^{[p]}$ lies in Kazhdan degree $ip + 2$. Hence $x^p - x^{[p]}$ lies in Kazhdan degree $ip + 2p$ and the symbol is $x^p \in \k[\g^*]$.
        
        We show that $J_\chi \subseteq \gr I_\eta$ by considering three cases.
        
        Case $i < -2$. For $x\in \g(i)$ we have $\chi(x) = \eta(x) = 0$ by the second paragraph of the proof, and so $\gr (x^p - x^{[p]} - \eta(x)^p) = x^p = x^p - \chi(x)^p$.

        Case $i = -2$. In this case $x\in \g(-2)$ has $x^p$ and $\eta(x)^p$ in Kazhdan degree 0 whilst $x^{[p]}$ lies in lower Kazhdan degree. Therefore $\gr (x^p - x^{[p]} - \eta(x)^p) = x^p - \eta(x)^p = x^p - \chi(x)^p$.

        Case $i > -2$. For $x \in \g(i)$, the Kazhdan degree of $x^p$ is greater than that of $x^{[p]}$ and also the degree of $\eta(x)^p$. So $\gr(x^p - x^{[p]} - \eta(x)^p) = x^p = x^p - \chi(x)^p$.

        We have shown that the generators of $J_\chi$ lie in $\gr I_\eta$, and this completes the proof.
\end{proof}

\subsection{Good transverse slices}

Let $\v \subseteq \g$ be a Dynkin graded vector space complement to $[\g,e]$ inside $\g$. If we write $\chi := \kappa(e)$ then the {\it good transverse slice to $\O_\chi$} is the variety
$$\S_\chi := \chi + \kappa(\v).$$
Although this depends on a choice of $\v$ and $\kappa$, the key properties of $\S_\chi$ are independent of these choices.

The inclusion $\g^e\subseteq \g(\geq0)$ implies that $[\g,e] \supseteq \g(>0) := \bigoplus_{i>0} \g(i)$, and so $\v \subseteq \g(\le 0)$ by homogeneity. Furthermore, $e$ lying in $\g(2)$ shows that $\chi$ is supported in graded degree $-2$, which is Kazhdan degree $0$. 

Since $\g^e\subseteq \g(\geq0)$ and  $\bar\kappa : \g^e \times \v \to \k$ is nondegenerate, we see that every homogeneous element of $\v$ lies in negative Kazhdan degree. 
It follows that $\S_\chi$ is $\gamma_e(\k^\times)$-stable, and the induced action is contracting: $\chi \in \overline{\gamma_e(\k^\times) \eta}$ for all $\eta \in \g^*$, and $\S_\chi^{\gamma_e(\k^\times)} = \{ \chi \}$.
As a result, $\k[\S_\chi] = \bigoplus_{i \ge 0} \k[\S_\chi]_i$ is positively Kazhdan graded.

The main properties of the slice are explained by the following well-known lemma. The standard reference for parts (1) and (2) over $\C$ is \cite[\textsection 2]{GG}. They can be proved for standard reductive groups by similar arguments, with suitable adaptations as in \cite[\textsection 5.1]{GTmodular}. For a direct proof of (3) see \cite[Corollary 3.24]{Ta}, whilst (4) follows from (3) by general arguments.

\begin{Lemma}
\label{L:sliceproperties}
\begin{enumerate}
\setlength{\itemsep}{4pt}
\item $\k[\S_\chi]$ inherits a Poisson structure by Hamiltonian reduction
\begin{eqnarray*}
\S_\chi \cong \mu^{-1}(\chi) /\!\!/ G(<0)
\end{eqnarray*}
where the moment map $\mu : \g^* \to \g(<\!\!-1)^*$ is given by restriction.
\item The Poisson bracket lies in degree $-2$ 
$$\{\k[\S_\chi]_i , \k[\S_\chi]_j\} \subseteq \k[\S_\chi]_{i+j-2}.$$
\item $\g^* = T_\eta \O_\eta + T_\eta\S_\chi$ for all $\eta \in \S_\chi$.
\item $\dim (\O_\eta \cap \S_\chi) = \dim(\O_\eta) - \dim(\O_\chi)$ for all $\eta \in \S_\chi$.
\end{enumerate} 
\end{Lemma}

\subsection{Finite $W$-algebras}\label{ss:finWalg}

The {\em finite $W$-algebra} corresponding to $\g$ and $e$ is a filtered quantization of the good transverse slice, constructed by quantum Hamiltonian reduction
\begin{eqnarray}
U(\g,e) := (U(\g) / U(\g) \{x-\chi(x) \mid x\in \g(<\! -1)\})^{G(<0)}.    
\end{eqnarray} More details can be found in \cite[\textsection 4]{GTmodular}; in the language of {\em loc. cit.}, we take $\l=0$ for the compatible isotropic subspace of $\g(-1)$.

The {\it variety of 1-dimensional representations of $U(\g,e)$} is $\E(\g,e) := \Spec U(\g,e)^\ab$ where $U(\g,e)^\ab$ denotes the maximal abelian quotient of $U(\g,e)$, which is defined by $$U(\g,e)^\ab := U(\g,e)/([u,v] \mid u,v \in U(\g,e)).$$ The 1-dimensional representations of $U(\g,e)$ are classified by the closed points of $\E(\g,e)$.

Let $I_{\S_\chi} \unlhd Z_p(\g)$ be the defining ideal of $\S_\chi^{(1)} \subseteq (\g^*)^{(1)}$. Consider the reduced enveloping algebra over the slice
$$U_{\S_\chi}(\g) := U(\g) / U(\g) I_{\S_\chi}.$$ Note that every irreducible $U_{\S_\chi}(\g)$-module is an irreducible $U_\eta(\g)$-module for some $\eta\in\S_\chi$, and so by \eqref{eq:Premetstheorem} and Lemma~\ref{L:sliceproperties}(4) has dimension divisible by $p^{\frac{1}{2}\dim\O_\chi}$.

\begin{Lemma} \cite[\textsection\textsection 7, 9]{GTmodular}
\label{L:Walgebraproperties}
Let $d_\chi=\dim\O_\chi$. The following hold:
\begin{enumerate}
\setlength{\itemsep}{4pt}
\item The Kazhdan filtration on $U(\g)$ descends to $U(\g,e)$ and induces a positive filtration such that $\gr U(\g,e) \cong \k[\S_\chi]$ as graded Poisson algebras.
\item $U_{\S_\chi}(\g) \cong \Mat_{p^{\frac{1}{2}d_\chi}} U(\g,e)$.
\item There is a bijection between the set of 1-dimensional $U(\g,e)$-modules and the set of $p^{\frac{1}{2}d_\chi}$-dimensional $U(\g)$-modules with $p$-character in $\S_\chi$.
\end{enumerate}
\end{Lemma}

 \subsection{Primitive ideals}
\label{ss:primitiveideals}

The primitive ideals of an associative algebra are the annihilators of its simple modules. We write $\Prim U(\g)$ for the set of primitive ideals in $U(\g)$, and $\Irr U(\g)$ for the isoclasses of simple $\g$-modules.

The following lemma is well-known, but we record a proof for the reader's convenience.
\begin{Lemma}
\label{L:simplesprimitivesmaximals}
\begin{enumerate}
    \item The natural map $\Irr U(\g) \to \Prim U(\g)$ is a bijection.
    \item The primitive ideals are the same as the maximal ideals in $U(\g)$.
\end{enumerate}
\end{Lemma}
\begin{proof}
   Since every $V\in \Irr U(\g)$ is finite-dimensional (see \cite{JaLA}, for example), Jacobson's density theorem shows that $U(\g) / \Ann_{U(\g)} V \cong \End_\k(V)$ as $\k$-algebras and so, in particular, the quotient admits a unique simple module up to isomorphism. Therefore the map in (1) admits a left inverse, proving (1). The argument also shows that every primitive quotient is simple, hence (2).    
\end{proof}

Since $\k$ is algebraically closed, we may apply Schur's lemma to see that each simple module has a central character. 
This implies that, for $I \in \Prim U(\g)$, the intersection $I \cap Z_p(\g) \in \Spec Z_p(\g)$ is a  closed point.
For $X \subseteq \Spec Z_p(\g) = (\g^*)^{(1)}$ we write
$$\Prim_X U(\g) := \{ I \in \Prim U(\g) \mid I \cap Z_p(\g) \in X\},$$  and for $d \in \Z$ we write $$\Prim^d U(\g) := \{ I \in \Prim U(\g) \mid \codim_{U(\g)} I = d\},$$ and $$\Prim_X^d U(\g) := \Prim_X U(\g) \cap \Prim^d U(\g).$$

The bijection in Lemma~\ref{L:simplesprimitivesmaximals} restricts to
\begin{eqnarray*}
    \Prim_X U(\g)&\overset{1\text{-}1}{\longrightarrow } &\coprod_{\eta\in X}\Irr U_\eta(\g),\\
    \Prim^d U(\g)&\overset{1\text{-}1}{\longrightarrow } &\{L\in\Irr U(\g)\mid \dim(L)^2=d\},
\end{eqnarray*}
where the union runs over the closed points of $X$. It follows from Lemma~\ref{L:Walgebraproperties}(3) that there is a bijection
\begin{eqnarray}
\label{eq:1dimsandprimitives}
\E(\g,e) \overset{1\text{-}1}{\longrightarrow }\Prim_{\S_\chi}^{p^{\dim \O_\chi}} U(\g).
\end{eqnarray}

\section{Coadjoint centralisers}\label{s:centralisers}

Throughout this section we retain the assumptions from Section~\ref{S:Walgebrassmallmodules}, though from mid-Subsection~\ref{ss:sheets} onward we restrict our attention to $G=\GL_N$. Our immediate goal is to prove some basic results about coadjoint centralisers which will be used later in the paper.

\subsection{Degenerations of coadjoint centralisers}
Let $e\in \Nc(\g)$. We equip $\g = \bigoplus_i \g(i)$ with the Dynkin grading, and $\k[\g^*] = \bigoplus_i \k[\g^*]_i$ with the induced grading. We define the (ascending) Dynkin filtration on $\k[\g^*]$ by $\k[\g^*]_{\le d} := \bigoplus_{i\le d} \k[\g^*]_i$, and we transfer this to a filtration on $\k[\g]$ via the isomorphism $\kappa : \g \to \g^*$. This defines a quotient filtration on $\k[V]$ for any subspace $V \subseteq \g$.

We similarly define the (ascending) Dynkin filtration on $\g = \bigcup_{d\ge 0} \g(\le d)$ and give any subspace $V\subseteq \g$ and quotient space $\g/V$ the induced filtration.

\begin{Lemma}
\label{L:Liealgebracentraliserdegenerations}
    Let $x\in e + \g(\le 1)$ be such that $\dim \g^x = \dim \g^e$.
    Then $$\gr \k[\g^x] = \k[\g^e].$$
\end{Lemma}
\begin{proof}

    Since $\bar\kappa$ is $G$-invariant it is $\g$-invariant, and so $\bar\kappa([\g,x], \g^x) = 0$, implying $\kappa[\g,x] = (\g^x)^\perp$. If $(\kappa[\g,x]) \unlhd \k[\g]$ denotes the ideal generated by $\kappa[\g,x]$ then it follows that $(\kappa[\g,x]) = I_{\g^x}$ is the vanishing ideal of $\g^x$.

    Consider now the map $\ad(x):\g\to\g$; it is filtered so long as we shift the filtration on the codomain by $2$. It is easy to see that $\gr(\ad(x))=\ad(e):\g\to\g$. In particular, $[\g,e]=\im(\gr(\ad(x))\subseteq \gr(\im(\ad(x))=\gr[\g,x]$. The assumption $\dim\g^x=\dim\g^e$ implies $\dim[\g,e]=\dim[\g,x]=\dim\gr[\g,x]$, and thus $\gr[\g,x]=[\g,e]$.

    Since $\kappa$ respects the Dynkin grading, we have $\kappa[\g,e] = \kappa\gr [\g,x] = \gr \kappa [\g,x] \subseteq \gr I_{\g^x}$ which implies $\k[\g^e] \onto \gr \k[\g^x]$. The surjection is an equality because both algebras have the same Krull dimension and $\k[\g^e]$ is reduced.
\end{proof}

\subsection{Sheets}\label{ss:sheets}

The {\it $k$th rank stratum} of $\g^*$ is the subset $\{\eta \in \g^* \mid \dim \O_\eta = k\}$, and the {\em sheets} of $\g^*$ are the irreducible components of the rank strata. They were studied extensively in \cite{PS}, following Borho's famous treatise in characteristic zero \cite{Bo}. Note that both of these cited works study sheets of Lie algebras of arbitrary reductive groups, whilst we mostly focus on the case $G = \GL_N$.

Every sheet is locally closed and contains a unique nilpotent orbit \cite[Proposition~2.5]{PS}. In fact, for $\gl_N$, this gives a bijection between the sheets and nilpotent orbits. For a proof see \cite[Proposition~2.3]{Bon} and the remarks which follow (he attributes the first proof to Dale Peterson's PhD thesis). We record the key fact for later reference.

\begin{Lemma}
For $G = \GL_N$ and $\O_\chi \subseteq \Nc(\g^*)$ there is a unique sheet containing $\O_\chi$.
\end{Lemma}

From now on, assume $G=\GL_N$. For $\chi = \kappa(e) \in \Nc(\g^*)$ we write $\bbS_\chi$ for the unique sheet containing $\O_\chi$. The {\it Katsylo section of $\bbS_\chi$} is defined as
\begin{eqnarray}
\chi + X := \bbS_\chi \cap \S_\chi.
\end{eqnarray}

\begin{Lemma}
\label{L:dogwobble}
    \begin{enumerate}
        \item $\chi + X$ is the set of elements $\eta\in \S_\chi$ such that $\dim \O_\eta = \dim \O_\chi$. 
        \item $\Prim_{\S_\chi}^{p^{\dim \O_\chi}} U(\g)=\Prim_{\chi+X}^{p^{\dim \O_\chi}} U(\g).$
    \end{enumerate}
\end{Lemma}
\begin{proof}
For (1) it suffices to show that $\eta \in \S_\chi$ lies in $\bbS_\chi$ whenever $\dim \O_\eta = \dim \O_\chi$. Indeed, the contracting Kazhdan $\k^\times$-action $\gamma_e(\k^\times)$ from \eqref{eq:Kazhdanaction} has the property that $\chi \in \overline{\gamma_e(\k^\times) \eta}$. Then $\overline{\gamma_e(\k^\times) \eta}$ is an irreducible variety containing $\chi$ and $\eta$ over which all $G$-orbits have constant dimension. This shows that $\eta \in \bbS_\chi$.

Part (2) follows from part (1), along with Lemma~\ref{L:PremetSkryabin} and Lemma~\ref{L:sliceproperties}(4).
\end{proof}

We can now upgrade \eqref{eq:1dimsandprimitives} slightly using Lemma~\ref{L:dogwobble}. We record this for ease of reference.
\begin{eqnarray}
\label{eq:1dimsandprimitives2}
\E(\g,e) \overset{1\text{-}1}{\longrightarrow }\Prim_{\chi+X}^{p^{\dim \O_\chi}} U(\g)
\end{eqnarray}

\subsection{Katsylo's Theorem}
\label{ss:Katsylosection}
For this subsection, we fix $G=\GL_N$.

Recall that conjugating by $\lambda_e(\k^\times)$ induces a $\k^\times$-action on $\k[G]$, and we grade $\k[G]$ by the eigenspaces. This induces a filtration $\k[G] = \bigcup_{d\in \Z} \k[G]_{\le d}$ by $\k[G]_{\le d} = \bigoplus_{i\le d} \k[G]_i$. In Subsection~\ref{ss:gradedHamactions} we denoted this filtration by $\k[G]^{\lambda_e}$, but we omit the superscript in this subsection for ease of notation.
\begin{Lemma}
\label{L:coneisGchi}
    For $\eta \in \chi + X$ we have $$\gr \k[G^\eta] = \k[G^\chi].$$
\end{Lemma}
\begin{proof}
    Let $\delta\in \k[\g]$ be the determinant function and $\k[\g]_\delta$ the localisation, so that $\k[\g]_\delta = \k[G]$. For $x\in \g$ we shall abuse notation to also write $\delta$ for the restriction to $\g^x$, so that $\k[\g^x]_\delta = \k[G^x]$. The inclusion $\k[\g^x] \subseteq \k[G^x]$ is filtered with respect to the filtration on $\k[G]$ described above, and the Dynkin filtration on $\k[\g]$ studied in Lemma~\ref{L:Liealgebracentraliserdegenerations}.

    Note that $\eta \in \chi + X$ corresponds, under the isomorphism $\kappa : \g \to \g^*$, to an element $x\in e + \g(\le 1)$ with $\dim \g^e = \dim \g^x$. Therefore, Lemma~\ref{L:Liealgebracentraliserdegenerations} implies $\gr \k[\g^\eta] = \k[\g^\chi]$.

    To complete the proof, we invoke a general fact, for which we leave the proof to the reader. Let $\B = \bigcup_{i\in \Z} \B_i$ be a $\Z$-filtered algebra such that $\gr \B$ is an integral domain, $1\in \B_0 \setminus \B_{-1}$, and $\delta \in \B_d\setminus \B_{d-1}$ for some $d \in \Z$. Write $\bar \delta = \delta + \B_{d-1}$. Then $\gr(\B_\delta) \cong (\gr \B)_{\bar \delta}$.
\end{proof}

By identifying $\k[G/G^\eta]$ (resp. $\k[G/G^\chi]$) with $\k[G]^{G^\eta}$ (resp. $\k[G]^{G^\chi}$), we may equip it with the induced filtration (resp. grading) from $\k[G]$.

\begin{Corollary}
    \label{C:gradedHopfcentraliser}
        For $\eta \in \chi + X$ we have $\gr \k[G / G^\eta]\subseteq \k[G / G^\chi]$.
    \end{Corollary}

\begin{proof}
    Let $I_{G^\chi}, I_{G^\eta} \subseteq \k[G]$ denote the defining ideals of $\k[G^\chi]$ and $\k[G^\eta]$ respectively. By Lemma~\ref{L:coneisGchi} we have $\gr I_{G^\eta} = I_{G^\chi}$.

    Let $\Delta$ denote the coproduct on $\k[G]$. It follows from the definition (see \cite[I.2.10(2)]{JanRAGS}, for example) that  $m \in \k[G]^{G^\eta}$ if and only if $\Delta(m) - m\otimes 1 \in \k[G] \otimes I_{G^\eta}$. 
    
    The map $\Delta$ is graded with respect to the $\lambda_e$-eigenspace grading and therefore it is strictly filtered: if $m \in \k[G]_{\leq i} \setminus \k[G]_{\leq i-1}$ then $\Delta(m) \in (\k[G] \otimes \k[G])_{\leq i} \setminus (\k[G] \otimes \k[G])_{\leq i-1}$.

    From the previous paragraph it follows that if $m \in \k[G]_{\le d}$ and $\bar m = m + \k[G]_{\le d-1}$ then $$\Delta(\bar m) - \bar m \otimes 1= \gr ( \Delta(m) - m \otimes 1)   \in \k[G] \otimes \gr (I_{G^\eta}) = \k[G] \otimes I_{G^\chi}.$$
    Whence $\bar m \in \k[G]^{G^\chi}$. Note that, at this final step, we use that fact that $\k[G]$ has a filtered basis to deduce $\gr(\k[G] \otimes I_{\g^\eta}) = \gr \k[G] \otimes \gr I_{G^\eta}$ (cf. Section~\ref{ss:filteredgradedspaces}). 
\end{proof}

The following result is a particular case of a theorem of Katsylo \cite{Kat} which was used to describe geometric quotients of sheets. Im Hof \cite{Im} reproved the theorem in characteristic zero, in an algebro-geometric framework. Here we provide a new, algebraic proof which works for $\GL_N$ in all characteristics.
\begin{Theorem}
\label{T:Katsylo}
   If $\O \subseteq \g^*$ is a coadjoint orbit with:
   \begin{enumerate}
       \item[(i)] $\O \cap \S_\chi \ne \emptyset$;
       \item[(ii)] $\dim \O = \dim \O_\chi$.
   \end{enumerate}
   Then the intersection $\O \cap \S_\chi$ is (scheme-theoretically) a reduced point.
\end{Theorem}
\begin{proof}
    If $\eta \in \O$ then $G / G^\eta \hookrightarrow \g^* \supseteq \S_\chi$ thanks to \eqref{eq:quotientisorbit}, and so the intersection is equal to the fibre product: $\O \cap \S_\chi = (G/G^\eta) \times_{\g^*} \S_\chi$. We must demonstrate that $\k[G/G^\eta] \otimes_{\k[\g^*]} \k[\S_\chi]$ is a 1-dimensional vector space.

    If we let $\k^\times$ act on $G \times \g^*$ via $(g, \eta) \mapsto (\lambda_e(t) g \lambda_e(t)^{-1}, \gamma_e(t)\eta)$ and $\k^\times$ act on $\g^*$ via $\gamma_e$, 
    then the action map $G \times \g^* \to \g^*$ is $\k^\times$-equivariant. So the pullback $\k[\g^*] \to \k[G] \otimes \k[\g^*]$ is graded, hence also filtered.
    
    If $\eta \in \chi + X$ then the map $\k[\g^*] \to \k[G] \otimes \k[\g^*] \overset{\Id \otimes \operatorname{ev}_\eta}{\longtwoheadrightarrow} \k[G] \otimes \k = \k[G]$ is also filtered, where $\operatorname{ev}_\eta$ denotes evaluation at $\eta$. Then $\k[\g^*] \to \k[G]$ factors through $\k[G]^{G^\eta}$, and the map $\k[\g^*] \to \k[G]^{G^\eta}$ is the homomorphism associated to the left-hand factor of $G/ G^\eta \times_{\g^*} \S_\chi$.

    Now $\k[G/G^\eta] \otimes_{\k[\g^*]} \k[\S_\chi]$ is a filtered algebra, and we may consider the associated graded algebra. Since $\k[\S_\chi]$ is a graded algebra it has a filtered basis, and so by the remarks following \eqref{eq:gradedandtensors}, along with Corollary~\ref{C:gradedHopfcentraliser}, we have $$\gr(\k[G/G^\eta] \otimes_{\k} \k[\S_\chi]) = \gr(\k[G/G^\eta]) \otimes_{\k} \k[\S_\chi] \subseteq \k[G / G^\chi] \otimes_\k \k[\S_\chi].$$

    Descending to tensor products over $\k[\g^*]$ we obtain
    \begin{eqnarray}
    \label{eq:Katsylomaps}
        \gr(\k[G/G^\eta] \otimes_{\k[\g^*]} \k[\S_\chi]) \twoheadleftarrow \gr(\k[G/G^\eta]) \otimes_{\k[\g^*]} \k[\S_\chi] \subseteq \k[G / G^\chi] \otimes_{\k[\g^*]} \k[\S_\chi].
    \end{eqnarray}

    Now we have $\k[G / G^\chi] \otimes_{\k[\g^*]} \k[\S_\chi] = \k[\pi^{-1}(\chi)]$ where $\pi : G / G^\chi \to \O_\chi$. By \cite[§2.2 \& Theorem~2.5]{JaNO} we see that $\pi$ is an isomorphism, hence $\k[\pi^{-1}(\chi)] = \k$.
    
    Finally, by \eqref{eq:Katsylomaps} we conclude that $\gr(\k[G/G^\eta] \otimes_{\k[\g^*]} \k[\S_\chi])$ is a quotient of a subspace of a 1-dimensional vector space. Together with assumption (i) we deduce that $\dim \gr(\k[G/G^\eta] \otimes_{\k[\g^*]} \k[\S_\chi]) =  1$, whence $\dim \k[G/G^\eta] \otimes_{\k[\g^*]} \k[\S_\chi] = 1$ by \eqref{eq:Katsylomaps}. This concludes the proof.
\end{proof}

\section{The general linear Lie algebra}
\label{S:generallinearalgebra}

In this section, $G=\GL_N$ and $\g=\gl_N$ for some $N>0$. Our goal is to show that the variety $\E(\g,e)$ can be identified with the variety $\z^*/W(\g_0)_\bullet$ from the Main Theorem. 

Throughout this section we fix $n > 0$ and let $\blambda = (\blambda_1,\blambda_2,...,\blambda_n) \vdash N$ be a partition, ordered so that $0 < \blambda_1 \le \blambda_2 \le \cdots \le \blambda_n$.

\subsection{Pyramids and tableaux} \label{ss_pyrtab}

A {\em pyramid $\bpi$ of shape $\blambda$} is a diagram with $\blambda_1$ boxes on the top row, $\blambda_{2}$ boxes on the row beneath, and so forth. They should be arranged so that each box which is strictly above the bottom row lies directly above a box in the row beneath. The boxes in a pyramid should be numbered $1,\ldots,N$, increasing by increments of one along rows and increasing down columns. For example, there are precisely two pyramids of shape $(1,2)$, as depicted below.
\begin{equation}
\label{e:somepyramids}
\begin{array}{c}
\begin{picture}(24,24) \put(0,0){\line(1,0){24}}
\put(0,12){\line(1,0){24}} \put(0,24){\line(1,0){12}}
\put(0,0){\line(0,1){24}} \put(12,0){\line(0,1){24}}
\put(24,0){\line(0,1){12}} 
\put(3,14.5){\hbox{1}}
\put(3,2.5){\hbox{2}}
\put(15,2.5){\hbox{3}}
\end{picture}
\end{array}
,\:\:\:\:
\begin{array}{c}
\begin{picture}(24,24) \put(0,0){\line(1,0){24}}
\put(0,12){\line(1,0){24}} \put(12,24){\line(1,0){12}}
\put(0,0){\line(0,1){12}} \put(12,0){\line(0,1){24}}
\put(24,0){\line(0,1){24}} 
\put(15,14.5){\hbox{1}}
\put(3,2.5){\hbox{2}}
\put(15,2.5){\hbox{3}}
\end{picture}
\end{array}.
\end{equation}

The rows of $\bpi$ are labeled $1,2,...,n$ from top to bottom and the columns of $\bpi$ are labeled $1,2,...,\blambda_n$ from left to right.
We denote the height of the $i$-th column by $q_i$, and we write $\row(i)$ and $\col(i)$ to denote the row number and the column number of $i$ in $\bpi$.
For example, in the left-hand pyramid associated to $(1,2)$ above, we have $q_1 = 2, q_2 = 1$ and $\row(1) = \col(1) = 1$.
For each row $r = 1, \dots, n$ of $\bpi$, define the set of \textit{capitals} of $r$ as 
$$C_r := \{ i \mid \row(i) = r, q_{\col(i)}=n-r+1 \}$$
and we partition each set of boxes in the $r$-th row into $C_r$ and the complement $C_r^c$.
The left pyramid in the above example has capitals $C_1 = \{1\}$ and $C_2 = \{3\}$, while for the right one $C_1 = \{1\}$ and $C_2 = \{2\}$ holds.

We identify the Weyl group $W$ of $\g$ with the symmetric group on $N$ letters in the usual manner.
This acts faithfully on the set of boxes of $\bpi$ by permutations.
We introduce two subgroups which are relevant for our purposes.
For each pair $(i,j)$ with $1\le i,j\le \blambda_n$ satisfying $q_i = q_j$ we let $w_{(i;j)} \in W$ denote the element which swaps the $i$-th and $j$-th column of $\bpi$.
We call such an element a {\em column-swap}, and write $W_{\col} \subseteq W$ for the subgroup generated by column-swaps.
It is clear that $W_{\col}$ induces an action on set of the capitals $C_r$ for each row $r$ of $\bpi$.
We also introduce the subgroup $W_{\row} \subseteq W$ which is the parabolic subgroup of $W$ consisting of all permutations $w \in W$ such that $\row(w(i)) = \row(i)$ for all $1 \leq i \leq N$.

Fix now a choice of pyramid $\bpi$ for $\blambda$. A {\it tableau of shape $\bpi$} is a filling of the boxes of $\bpi$ with elements of $\k$.
The set of tableaux of shape $\bpi$ is denoted $\Tab_\k(\bpi)$.
If $\bA \in \Tab_\k(\bpi)$ then we write $a_i$ for the entry of the $i$-th box. 
The group $W$  acts on $\Tab_\k(\bpi)$ such that $w\in W$ maps the tableau $\bA$ to the tableau $w(\bA)$ with entries $a_{w(i)}, 1 \leq i \leq N$.

If  $\bA, \bA' \in \Tab_\k(\bpi)$  lie in the same $W_{\col}$-orbit (resp. in the same $W_{\row}$-orbit), we say that they are {\it column-swap equivalent} (resp. that they are {\it row equivalent}).
It is clear that if $\bA$ and $\bA'$ are column-swap equivalent then they are row-equivalent.
We now introduce a special subset of tableaux for which the converse implication also holds.

A tableau $\bA \in \Tab_\k(\bpi)$ is called {\it column connected} if $a_i = a_j + 1$ whenever the $i$-th box lies directly above the $j$-th box in $\bpi$. 
The subset of column connected tableaux will be denoted $\Tab^{\cc}_\k(\bpi)$; this subset is clearly $W_{\col}$-stable.

\begin{Lemma} \label{L:rowiffcol}
    Let $\bA, \bA' \in \Tab^{\cc}_\k(\bpi)$. Then $\bA$ and $\bA'$ are row-equivalent if and only if $\bA$ and $\bA'$ are column-swap equivalent. 
\end{Lemma}
\begin{proof}
    Since $\bA$ is column-connected, for all $r = 1,  \dots, n$ and all $i \in C_r^c$ there exists $k\in\{1,\ldots,N\}$ with $\row(k) = r-1$, $\col(k) = \col(i)$, and $ a_i = a_k + 1$; similarly for $\bA'$.
    
    By assumption, there exists $w \in W_{\row}$  such that $a'_i = a_{w(i)}$ for all $i =1, \dots, N$.
    We give an algorithm producing $u \in W_{\row}$ such that $a'_i = a_{u(i)}$ for all $i =1, \dots, N$ and $u(C_r) = C_r$ for all $r = 1, \dots, n$. Let $r$ be minimal such that $w(C_r)\neq C_r$; note that $r>1$ since $C_1^c=\emptyset$. Since $w(C_r) \neq  C_r$, there exists a box $i$ with $\row(i) =r, i \notin C_r$ and $w(i) \in C_r$. Then there exists $k\in\{1,\ldots,N\}$ such that $\row(k) = r-1$ and $\col(k) = \col(i)$, and this satisfies $a_{w(i)} =a'_i= a'_k +1 = a_{w(k)} +1$.
    But there also exists a box $j$ with $\row(j) = r$ and $ \col(j) = \col(w(k))$, and this satisfies $j\in C_r^c$ (so $j\neq w(i)\in C_r$) and $a_j =a_{w(k)}+1= a_{w(i)}$.

\begin{center}

\begin{picture}(150,40)

\put(0,0){\line(1,0){20}}
\put(0,0){\line(0,1){40}}

\put(0,20){\line(1,0){20}}
\put(0,40){\line(1,0){20}}

\put(20,0){\line(0,1){40}}
\put(40,0){\line(1,0){20}}
\put(40,0){\line(0,1){40}}

\put(60,0){\line(0,1){40}}
\put(40,40){\line(1,0){20}}
\put(40,20){\line(1,0){20}}

\put(80,0){\line(1,0){20}}
\put(80,0){\line(0,1){20}}
\put(80,20){\line(1,0){20}}
\put(100,20){\line(0,-1){20}}

\put(5,25){\hbox{$k$}}
\put(5,5){\hbox{$i$}}

\put(23,25){\hbox{$\cdots$}}
\put(23,5){\hbox{$\cdots$}}

\put(38,25){\hbox{ \tiny $w(k)$}}
\put(45,5){\hbox{$j$}}

\put(63,5){\hbox{$\cdots$}}

\put(85,28){\hbox{\large $\times$}}
\put(78,5){\hbox{ \tiny $w(i)$}}

\put(110,25){\hbox{$\leftarrow$ row $r-1$}}
\put(110,5){\hbox{$\leftarrow$ row $r$}}
\end{picture}

\end{center}

    Set $\bar w:= (j \; w(i)) w$. 
    Then $a_{w(i)} = a_{\bar w(i)}$ for all $i = 1, \dots, N$, and so $\bA'=w(\bA)=\o{w}(\bA)$.
    
    If $l \in C_r$ satisfies $w(l) \in C_r$, then also $\bar w (l) \in C_r$; moreover $\bar w (i) \in C_r$. Therefore, $\lvert C_r \cap \bar w^{-1} (C_r)) \rvert > \lvert C_r \cap w^{-1} (C_r)) \rvert$.
    We then iterate the procedure, substituting $w$ with $\bar w$; the inequality in the previous line ensures the condition $w(C_r) = C_r$ is reached in a finite number of steps.
    (Observe that  if $w(k) = k$ then $i  = j$ and the argument still works.)

    The final permutation obtained by running the algorithm is our sought $u$.
    Since $\bA,\bA'$ are column-connected, the fillings of $\bA$, resp. $\bA'$, are uniquely determined by the values $a_i$, resp.  $a'_i$ for $i \in \bigcup_{r = 1}^n C_r$.
    Since $u$ acts permuting the sets $C_r$, we conclude that $u$ swaps the columns of $\bA$ of the same height, whence the thesis.
\end{proof}

\subsection{Subalgebras associated to a pyramid}
\label{ss:subalgebrasassociated}
The standard matrix units in $\g$ will be denoted $e_{i,j}$, and we write $\t = \sspan\{e_{i,i} \mid i=1,...,N\}$ for the diagonal torus. There is a nice choice of $G$-equivariant isomorphism $\kappa : \g \to \g^*$ given by $X \mapsto (Y \mapsto \Tr(XY))$, which restricts to $\kappa : \t \overset{\sim}{\to} \t^*$. 

Associated to the pyramid $\bpi$, there is a parabolic subalgebra $\p \subseteq \gl_N$ with Levi factor $\g_0$ and nilradical $\r$ defined by
\begin{eqnarray*}\p &=& \sspan\{e_{i,j} \mid \col(j) \ge \col(i)\},\\
\g_0 &=& \sspan\{ e_{i,j} \mid \col(i) = \col(j)\},\\\r &=& \sspan\{e_{i,j} \mid \col(j) > \col(i)\}.
\end{eqnarray*}

The centre of the Levi is equal to
$\z(\g_0) = \sspan\{ \sum_{\col(i) = k} e_{i,i} \mid k=1,...,\blambda_n\},$
and this identifies under $\kappa$ with
$$\z^* := \Ann_{\t^*} ([\g_0, \g_0]\cap \t) \subseteq \t^*.$$

We consider the following Borel subalgebra of $\g$
$$\b = \sspan\{e_{i,j} \mid \row(i) < \row(j), \text{ or } \row(i) = \row(j) \text{ and } \col(i) < \col(j)\},$$
and we write $\rho \in \t^*$ for the half-sum of positive roots corresponding to $\b$ (note that in \cite{GT} this is denoted $\widetilde{\rho}$). We remark that $\b \not\subset \p$, that the choice of $\b$ is not the upper triangular matrices, and that $\rho\notin \z^*$, in general.

We define a nilpotent element $e\in \gl_N$ by
$$e = \sum_{\substack{\row(i) = \row(j)\\ \col(i) = \col(j)-1}} e_{i,j},$$
and we write $\chi =\kappa(e) \in \g^*$. Observe that $\O_\chi$ corresponds to the partition $\lambda$ under the bijection between nilpotent coadjoint $\GL_N$-orbits and partitions of $N$. There is a grading of $\g$ associated to $\bpi$, which is defined by
$$\g(d) = \sspan\{e_{i,j} \mid \col(j) - \col(i) = d\}.$$
Note that $\p = \g(\ge 0)$.

Recall from the Introduction that $W(\g_0):=N_G(\g_0)/G_0$ is the relative Weyl group of $\g_0$, where we identify $W$ with permutation matrices in $\GL_N$ in order to act on $\g_0$, and where $G_0$ is the Levi subgroup of $\g$ with $\Lie(G_0)=\g_0$. We remark that it is straightforward to identify $W(\g_0)$ with the subgroup $W_{\col}\subseteq W$ generated by column swaps, and we do so for the remainder of this section. Then $W(\g_0)$ acts on $\t^*$ via the dot-action $w_\bullet \zeta = w(\zeta + \rho) - \rho$. 

For $i=1,\ldots,N$, denote by $\varepsilon_i\in\t^*$ the map $e_{jj}\mapsto \delta_{ij}e_{ii}$.
Given $\bA \in \Tab_\k(\bpi)$, we define $\lambda_\bA := \sum{a_i}\varepsilon_i \in \t^*$, and then define a bijection $\Tab_\k(\bpi)\xrightarrow{\sim}\t^*$, $\bA \mapsto \lambda_\bA-\rho$. This map is $W(\g_0)$-equivariant, where $W(\g_0)$ acts on $\Tab_\k(\bpi)$ as in Subsection~\ref{ss_pyrtab} and on $\t^*$ via the dot-action. As in \cite[\textsection 2.5]{GT}, $\Tab_\k^{\cc}(\bpi)$ corresponds to $\z^*$ under this bijection; in particular, this means that $\z^*$ is preserved by $W(\g_0)$ under the dot-action.

\subsection{The Miura map and pullback of representations}

The finite $W$-algebra $U(\g,e)$ can be realised as a subalgebra of the enveloping algebra $U(\p)$ (see  \cite[§2.6]{GT}) and since $\g_0$ is  a Levi factor of $\p$, the natural projection $\p \twoheadrightarrow \g_0$ allows to  consider
    the  composed map
$$\mu : U(\g,e) \to U(\p) \to U(\g_0),$$
 known as the Miura map.
If $M \in U(\g_0)\lmod$ then we can pullback $M$ to $U(\g,e)\lmod$ along $\mu$. 
Since $\z^*$ parameterises 1-dimensional representations of $\g_0$ we obtain a map
\begin{eqnarray}
\label{eq:Miurapullback}
\tilde\mu : \z^* \to \E(\g,e).
\end{eqnarray}

\begin{Theorem}
\label{T:OnedimensionalsviaMiura}
\begin{enumerate}
    \item The map \eqref{eq:Miurapullback} is a quotient mapping for the action $W(\g_0)_\bullet \curvearrowright \z^*$.

    \item We have a natural isomorphism of varieties
   $$\z^* / W(\g_0)_\bullet \isoto \E(\g,e).$$
\end{enumerate}
\end{Theorem}

\begin{proof}
By the proof of \cite[Theorem 5.1]{GT}
 there is a surjective map $\Tab^{\cc}_\Bbbk(\bpi) \onto \E(\g, e)$ associating to each  column-connected tableau $\bA$  the representation which in  \emph{op. cit.} is denoted $\widetilde{\Bbbk}_{\bar \bA}$.
 Moreover, since  $\widetilde{\Bbbk}_{\bar \bA}$ is determined by means of symmetric polynomials in the entries of the rows of $\bA$, the fibre of $\widetilde{\Bbbk}_{\bar \bA}$ through this map is exactly $W_{\row} \bA \cap \Tab^{\cc}_\Bbbk(\bpi)$.
From  the first part of the proof of \emph{loc. cit.} we deduce  that $\widetilde{\Bbbk}_{\bar \bA}$ is the pull-back to $U(\g, e)$ via $\mu$ of the 1-dimensional $U(\g_0)$-module with weight $\lambda_\bA - \rho$.
 
Summing up, we have a commutative diagram:
\begin{center}
\begin{tikzcd}
 \Tab^{\cc}_\Bbbk(\bpi)  \arrow[rd] \arrow[d, rightarrow] & \\
    \z^*  \arrow[r, "\tilde \mu"] &\E(\g,e)
\end{tikzcd}
\end{center}
As noted above, the vertical arrow is $W(\g_0)$-equivariant: the action on the domain is introduced in §\ref{ss_pyrtab} and on the codomain is the dot action of $W(\g_0)$.
Moreover, the upper-right arrow is $W(\g_0)$-invariant by above considerations on the fibres of the map and Lemma \ref{L:rowiffcol} (recall that $W_{\col}$ identifies with $W(\g_0)$).

This implies that $\tilde \mu$ fulfills the conditions of a geometric quotient. Firstly, $\z^*$ is an affine space and $W(\g_0)$ is reductive. Furthermore, $\E(\g,e)$ is an affine space by Lemma 3.2 in \emph{op. cit.}. The fibres of the surjective map $\tilde \mu$ are $W(\g_0)_\bullet$-orbits. Finally, the pull-back of the map $\tilde \mu$ is an isomorphism of algebras $U(\g,e)^{\ab} \simeq \k[\z^*]^{W(\g_0)_\bullet}$. Indeed, the algebra $U(\g,e)^{\ab}$ is a polynomial algebra and thus is reduced, see \emph{loc. cit.}. It is easy to check using the arguments in Section 5 of \emph{op. cit.} that each generator is mapped to a ${W(\g_0)_\bullet}$-invariant polynomial in the coordinates of $\z^*$ and that the resulting map is an isomorphism.
\end{proof}

\section{Study of quantizations}\label{s:Quants}

In this section we prove the main theorem of the paper, which appears here as Theorems~\ref{T:MainThm} and \ref{C:quantshavecharacters}. Our main goal is a bijection $$\Idl_{\o{\O}_\chi}^{G} U(\g)\xlongleftrightarrow{1-1} \Prim_{\chi+X}^{p^{\dim\O_\chi}} U(\g).$$ Here, the notation $\Idl_{\o{\O}_\chi}^{G} U(\g)$ is as in Subsection~\ref{ss:quantclosedsubschemes} and the notation $\Prim_{\chi+X}^{p^{\dim\O_\chi}} U(\g)$ is as in Subsection~\ref{ss:primitiveideals}. Unless otherwise stated we work with the PBW filtration, though the reader should be on notice that we do sometimes otherwise state.

Throughout, we take $G=\GL_N$, although some of our proofs do not require this -- see Remark~\ref{R:whyGLN} for a full discussion of this assumption. We fix $\chi\in\Nc(\g^*)$, and write $e=\kappa^{-1}(\chi)$ and $\O_\chi=G\cdot\chi$. Recall the notation $d_\chi:=\dim\O_\chi$.

\subsection{The $p$-support of a quantization}\label{ss:psupport}

For $M \in U(\g)\lmod$, we define the {\em $p$-support} $\Supp_p(M)$ to be the support of $M|_{Z_p(\g)} \in Z_p(\g)\lmod$, i.e. the set of $I\in\Spec Z_p(\g)$ such that the localisation $M_I$ is non-zero. Since $Z_p(\g)$ is finitely generated and noetherian, it is well known that the $p$-support may equivalently be defined to be the closed subscheme of $(\g^*)^{(1)}$
 $$\Supp_p(M) := \{ I \in \Spec Z_p(\g) \mid \Ann_{Z_p(\g)}M\subseteq I\}.$$
We can describe the closed points of $\Supp_p(M)$ in the following useful way. Recall the notation $I_\eta \in \Spec Z_p(\g)$ from \eqref{eq:idealIchi}.
\begin{Lemma}
\label{L:whatissupport}
The closed points of $\Supp_p(M)$ can be identified with the $\eta \in \g^*$ such that $I_\eta M \ne M$.
\end{Lemma}
We will use this lemma repeatedly in what follows, and we always identify the $p$-support with its set of closed points.

Let $I \in \Idl_{\overline{\O}_\chi}^G U(\g)$ and set $\A := U(\g)/ I$, so that $\gr \A = \k[\o\O_\chi]$. We regard $\A$ as a $U(\g)$-module via the natural map $\Phi : U(\g) \onto \A$. The {\it $p$-centre of $\A$} is defined to be $Z_p\A := \Phi Z_p(\g)$. To describe the $p$-support of $\cA$, we first need the following lemma about the Poisson structure of $\k[\o\O_\chi]$.
\begin{Lemma}
\label{L:casimirspthpowers}
The Poisson centre $Z\k[\cO_\chi]$ is equal to $\k[\cO_\chi]^p$, and the centre $Z\A$ is equal to $Z_p\A$.
\end{Lemma}
\begin{proof}
Since $\k[\cO_\chi]=\k[\O_\chi]$ by Lemma~\ref{L:normalorbits}, we can apply \cite[Lemma~1.10]{BeKa} to $X=\O_\chi$ and take global sections. This proves the first claim. Taking the associated graded morphism of the inclusion $Z_p\cA\subseteq Z\cA$ yields inclusions $$ \k[\cO_\chi]^p=(\gr\Phi)(\gr Z_p(\g))\subseteq \gr(Z_p\cA)\subseteq \gr(Z \cA)\subseteq Z\k[\cO_\chi]=\k[\cO_\chi]^p,$$ and thus that $\gr(Z_p\cA)=\gr(Z\cA)$. This implies the result.
\end{proof}

Our next result shows that the $p$-support of $\cA$ is the closure of a (not necessarily nilpotent) coadjoint $G$-orbit. Recall that $\bbS_\chi$ denotes the (unique) sheet of $\g^*$ containing $\O_\chi$ (see Section~\ref{ss:sheets}).

\begin{Proposition}
\label{P:psupport}
    $\Supp_p(\A) = \overline{\O}^{(1)}$ for some $G$-orbit $\O \subseteq \bbS_\chi$. 
\end{Proposition}
\begin{proof}

For $J \lhd Z_p(\g)$ let $\bbV(J) = \{ J' \in \Spec Z_p(\g) \mid J' \supseteq J\}$, viewed as a reduced scheme. 
Denote $I_p:=I\cap Z_p(\g)\lhd Z_p(\g)$
so that $\Supp_p(\A) = \bbV(I_p)$
and identify $Z_p(\g)$ with $\k[(\g^*)^{(1)}]$. 
Throughout the proof we use the standard grading on $\k[\g^*]$ placing $\g$ in degree 2, and regard $\k[\g^*]$ as a filtered algebra by $\k[\g^*] = \bigcup_{d \ge 0} \k[\g^*]_{\le d}$. 
The identification $\k[(\g^*)^{(1)}]\isoto \k[\g^*]^p\subseteq \k[\g^*]$ gives a corresponding grading and filtration on $\k[(\g^*)^{(1)}]$. These choices allow us to consider the associated graded ideal $\gr J \lhd \k[(\g^*)^{(1)}]$ for any ideal $J \lhd \k[(\g^*)^{(1)}]$, and define the asymptotic cone $\k \bbV(J) := \bbV(\gr J)$, a reduced subscheme of $(\g^*)^{(1)}$.

{\bf Claim 1:} $\k \bbV(I_p)=\o{\O}_\chi^{(1)}$.

{\bf Proof of Claim 1:} 
Denote by $\nu$ the composition $Z_p(\g)\hookrightarrow U(\g)\twoheadrightarrow \cA$, and equip $Z_p\cA=\nu(Z_p(\g))\subseteq \cA$ with the subspace filtration (where $U(\g)$ has the PBW filtration with $\g$ in degree $2$). By Lemma~\ref{L:casimirspthpowers}, $\gr(\nu(Z_p(\g)))=\k[\o{\O}_\chi]^p$ and thus $\gr(Z_p(\g))\to\gr(\nu(Z_p(\g)))$ is surjective. By \cite[Corollary 6.14]{MR}, this implies that $Z_p(\g)\to \nu(Z_p(\g))$ is surjective and strict. In particular, the image and coimage of this map are isomorphic as filtered algebras, i.e. $Z_p(\g)/I_p\cong \nu(Z_p(\g))$, and thus $\gr(Z_p(\g)/I_p)\cong \k[\o{\O}_\chi]^p=\k[\o\O_\chi^{(1)}]$, as required.

{\bf Claim 2:} $\bbV(I_p)$ is an orbit closure in $(\g^*)^{(1)}$.

{\bf Proof of Claim 2:} Let $(\g^*)^{(1)}_{\leq k}$ denote the union of orbits of dimension less than $k+1$. 
Since $\bbV(I_p)$ is $G$-stable it is a union of orbits. 
We claim that $\dim\O_\chi$ is the maximal dimension of $G$-orbits in $\bbV(I_p)$. 
If not, let $k<\dim\O_\chi$ be such that $\bbV(I_p)\subseteq (\g^*)^{(1)}_{\leq k}$. 
Since the rank strata are $\k^\times$-stable we may apply \cite[5.3(4), (5)]{PrST} to deduce that $\bbV(\gr I_p) \subseteq (\g^*)^{(1)}_{\leq k}$. 
This contradicts $\bbV(\gr I_p) = \o\O_\chi^{(1)}$. Since $\dim(\bbV(I_p))=\dim(\k \bbV(I_p))=\dim\O_\chi$, 
$\bbV(I_p)$ must therefore contain an orbit of dimension $\dim\O_\chi=\dim(\bbV(I_p))$. 
Furthermore, $I_p$ is prime, because $Z_p(\g)/I_p$ embeds in $\A$ which has no zero divisors, and thus $\bbV(I_p)$ is irreducible. This proves the claim.

{\bf Claim 3:} $\bbV(I_p)=\o\O^{(1)}$ for some $G$-orbit $\O\subseteq \bbS_\chi$.

{\bf Proof of Claim 3:} From Claim 2, we know that $\bbV(I_p)=\o\O^{(1)}$ for some $G$-orbit $\O\subseteq \g^*$. 
Then \cite[5.3(5)]{PrST} implies that $\o\O_\chi^{(1)}=\k \bbV(I_p)\subseteq \overline{\k^\times\O^{(1)}}$ and $\dim \O_\chi^{(1)} = \dim \O^{(1)}$. Consider the locally closed set $D := \overline{\k^\times \O^{(1)}} \setminus (\g^*)^{(1)}_{\le \dim \O-1}$. 
We have shown that $\O_\chi^{(1)}$ and $\O^{(1)}$ are contained in $D$. Since $D$ is irreducible and consists of orbits of the same dimension, it is contained in some sheet of $(\g^*)^{(1)}$. This completes the proof.
\end{proof}

We collect here some results about the structure of $\cA$ as a $Z_p\A$-algebra. Recall that for an integral domain $R$, the rank of an $R$-module $M$ is $\dim_{F}(F\otimes_R M)$ where $F = \Fract(R)$ is the field of fractions.

\begin{Proposition}
\label{P:fibredimension}
The following hold:
\begin{enumerate}
    \setlength{\itemsep}{4pt}
    \item[(i)] $\A$ has rank $p^{d_\chi}$ over $Z_p\A$;
    \item[(ii)] If $\O \subseteq \g^*$ is the orbit of Proposition~\ref{P:psupport}, so that $\Supp_p(\A) = \o\O$, then for $\eta \in \O$ $$\dim \A/\Phi(I_\eta)\A = p^{d_\chi}.$$
\end{enumerate}
\end{Proposition}
\begin{proof}
    (i) It is well-known that the rank of a filtered module is preserved by taking the associated graded module under our hypotheses (see \cite[Lemma 2.3]{Ti}, for example). Therefore it suffices to show that $\gr(\A)=\k[\o\O_\chi]$ has rank $p^{d_\chi}$ over $\gr(Z_p\A)=\k[\o\O_\chi]^p$ or, equivalently, that $$[\Fract(\k[\o\O_\chi]):\Fract(\k[\o\O_\chi])^p]=p^{d_\chi}.$$ The latter follows easily from the fact that $\dim\O_\chi$ equals the transcendence degree of $\Fract(\k[\o\O_\chi])$ over $\k$.    
    
    (ii) By (i), $\Supp_p(\A)$ contains a $G$-stable open subset on which the desired equality holds. This subset must contain $\O$, which proves the result.
\end{proof}

\subsection{From quantizations to primitive ideals}
\label{ss:quantisationtoprimitiveideals}

We restate some notation. Let $\S_\chi$ be the good transverse slice to $\O_\chi$ at $\chi$. The unique sheet of $\g^*$ containing $\chi$ is denoted $\bbS_\chi$ and the Katsylo section is $\chi + X := \S_\chi \cap \bbS_\chi$.

We define a map
\begin{eqnarray}
\label{eq:ideal_map}
    \Theta : \Idl_{\o\O_\chi}^G U(\g) \to \Prim_{\chi + X}^{p^{d_\chi}} U(\g),
\end{eqnarray}
as follows. Fix $I \in \Idl_{\o\O_\chi}^G U(\g)$ and write $\A = U(\g) / I$. There is an orbit $\O \subseteq \bbS_\chi$ such that $\o \O^{(1)} = \Supp_p(\A)$, by Proposition~\ref{P:psupport}. Thanks to Theorem~\ref{T:Katsylo} there is a unique $\eta \in \O \cap \S_\chi\subseteq\chi+X$.

Write $\Phi : U(\g) \onto \A$ and recall that $\eta$ corresponds to a maximal ideal $I_\eta \subseteq Z_p(\g)$. We define
\begin{eqnarray}
    \label{eq:keyidealmap}
    \Theta(I) := \Ker\big(U(\g) \to \A / \Phi(I_\eta) \A\big) = I + I_\eta U(\g).
\end{eqnarray}

Since the maps $U(\g) \onto \A \onto \A / \Phi(I_\eta)\A$ are surjective we have $U(\g) / \Theta(I) \cong \A / \Phi(I_\eta)\A$. Together with Proposition~\ref{P:fibredimension} we see that $\codim_{U(\g)} \Theta(I) = p^{d_\chi}$.

It follows from Lemma~\ref{L:KW+Humphreys} that all maximal ideals of $U_\eta(\g)$ have codimension greater than or equal to $p^{\dim \O_\eta}$. 
Minimality of codimension  implies that $\Theta(I)$ (which contains $I_\eta$) is a maximal ideal of $U(\g)$ and hence it is primitive.

By construction $\Theta(I) \cap Z_p(\g) = I_\eta$, and so we have demonstrated that \eqref{eq:ideal_map} is well-defined. We record this as a lemma.
\begin{Lemma}\label{L:IdltoPrim}
    $\Theta(I) \in \Prim_{\chi + X}^{p^{d_\chi}} U(\g)$. $\hfill \qed$
\end{Lemma}

\subsection{The coordinate ring of an orbit as an induced module}\label{ss:coordringdesc}

For a subgroup scheme $H \subseteq G$ and $M \in H\lmod$ we use the standard notation for induction (see \cite[\textsection I.3.2]{JanRAGS} for more details):
\begin{eqnarray}
    \Ind_H^G(M) := (\k[G] \otimes M)^H \in G\lmod.
\end{eqnarray}
Here, $H$ acts on $\k[G]$ via the right regular representation and on $M$ via the given $H$-module structure, while $G$ acts on $\Ind_H^G(M)$ via the left regular representation on $\k[G]$ and via the trivial $G$-module structure on $M$. If $M$ is a $\k$-algebra and $H$ acts on $M$ via algebra automorphisms, then clearly $\Ind_H^G(M)$ has the structure of a $\k$-algebra. Similarly, if $M$ is a Poisson $\k$-algebra and $H$ acts on $M$ via Poisson automorphisms, then $\Ind_H^G(M)$ can be equipped with the structure of a Poisson $\k$-algebra with $\{f\otimes m,f'\otimes m'\}=ff'\otimes\{m,m'\}$.

Recall that by Lemma~\ref{L:PremetSkryabin} there is a unique maximal Poisson ideal $P_\chi \lhd \k_\chi[\g^*] = \k[\g^*] / J_\chi$, where $J_\chi = (x^p - \chi(x)^p \mid x\in \g)$. The quotient $\k_\chi[\g^*] / P_\chi$ is then a simple Poisson algebra of dimension $p^{d_\chi}$, on which $G^{F(\chi)}$ acts by Poisson automorphisms. As a $G^{F(\chi)}$-algebra, it has the following alternative description. 

\begin{Lemma}
\label{L:inducedsimplePoissonquotient}
    $\Ind_{G^\chi}^{G^{F(\chi)}}(\k) \isoto \k_\chi[\g^*]/P_\chi$ as $G^{F(\chi)}$-algebras.
\end{Lemma}
    \begin{proof}

    We prove in three steps that such an isomorphism of $G^{F(\chi)}$-algebras exists. First we recall some notation: we set $\mu^*:\k[\g^*]\twoheadrightarrow\k[\o\O_\chi]$ to be the comoment map, and $F_{\o\O_\chi}:\o\O_\chi\to\o\O_\chi^{(1)}$ the Frobenius morphism on $\o\O_\chi$.

    \noindent{\bf Step 1:} $\k[\o\O_\chi]/\mu^*(J_\chi)=\k[F_{\o\O_\chi}^{-1}F_{\o\O_\chi}(\chi)].$

    Since $\mu^*(J_\chi)$ is generated by $p$-th powers of generators of the defining ideal of the closed point $\chi\in\o\O_\chi$, this is \cite[I.9.3]{JanRAGS} applied to $X=\cO_\chi$.

    \noindent{\bf Step 2:} $\k_\chi(\g)/P_\chi\iso \k[\o\O_\chi]/\mu^*(J_\chi)$.

   The Poisson surjection $\k[\g^*]\twoheadrightarrow \k[\cO_\chi]/\mu^*(J_\chi)$ factors through $\k_\chi[\g^*]\twoheadrightarrow \k[\cO_\chi]/\mu^*(J_\chi)$. By Step 1 and \cite[I.9.3(3)]{JanRAGS}, $\dim(\k[\cO_\chi]/\mu^*(J_\chi))=p^{d_\chi}$. The claim follows from Lemma~\ref{L:PremetSkryabin}.

    \noindent{\bf Step 3:} $\k[F_{\o\O_\chi}^{-1}F_{\o\O_\chi}(\chi)]=\Ind_{G^\chi}^{G^{F(\chi)}}(\k)$.
   
   Since $F_{\o\O_\chi}^{-1}F_{\o\O_\chi}(\chi)\subseteq \O_\chi$, we easily obtain that $G^{F(\chi)}$ acts transitively on it and thus that $F_{\o\O_\chi}^{-1}F_{\o\O_\chi}(\chi)\iso G^{F(\chi)}/G^\chi$. Thus, $$\k[F_{\o\O_\chi}^{-1}F_{\o\O_\chi}(\chi)]=\k[G^{F(\chi)}/G^\chi]=\k[G^{F(\chi)}]^{G^\chi}=\Ind_{G^\chi}^{G^{F(\chi)}}(\k).$$

   Combining Steps (1)--(3) yields the isomorphism of $G^{F(\chi)}$-algebras. 

\end{proof}
For the following result, give $\Ind_{G^{F(\chi)}}^G (\k_\chi[\g^*] / P_\chi)$ the Poisson algebra structure induced from that of $\k_\chi[\g^*] / P_\chi$.

\begin{Corollary}
\label{C:coordasinduced}
    $\Ind_{G^{F(\chi)}}^G (\k_\chi[\g^*] / P_\chi) \isoto \k[\O_\chi]$ as Poisson $G$-algebras.
\end{Corollary}
\begin{proof}
    By \cite[I.3.5(2) \& I.5.6(2)]{JanRAGS} and Lemma~\ref{L:inducedsimplePoissonquotient} we have
    \begin{equation}\label{eq:coordeq}
    \k[\O_\chi] = \k[G / G^\chi] = \Ind_{G^\chi}^G(\k) = \Ind_{G^{F(\chi)}}^G \Ind_{G^\chi}^{G^{F(\chi)}} (\k) = \Ind_{G^{F(\chi)}}^G (\k_\chi[\g^*] / P_\chi)
    \end{equation}
    as $G$-algebras. 
    
    What remains is to show that the composition is a Poisson homomorphism. Frobenius reciprocity yields bijections $$\Hom_G(\k[\O_\chi],\Ind_{G^{F(\chi)}}^G(\k_\chi[\g^*]/P_\chi))\simeq\Hom_{G^{F(\chi)}}(\k[\O_\chi],\k_\chi[\g^*]/P_\chi)\simeq\Hom_{G^{F(\chi)}}(\k[\O_\chi],\k[\O_\chi]/\mu^*(J_\chi)).$$ The claim then follows from two observations: (1) Frobenius reciprocity identifies Poisson homomorphisms with Poisson homomorphisms, and (2) the identification \eqref{eq:coordeq} corresponds under Frobenius reciprocity to the natural projection $\k[\O_\chi]\twoheadrightarrow \k[\O_\chi]/\mu^*(J_\chi)$ (which is a Poisson homomorphism). The first observation follows from the fact that $G$ acts on $\k[\O_\chi]$ by Poisson automorphisms, and the second is straightforward to check using Lemma~\ref{L:inducedsimplePoissonquotient}.
\end{proof}

We now equip $\Ind_{G^{F(\chi)}}^G (\k_\chi[\g^*] / P_\chi)$ with a grading such that Corollary~\ref{C:coordasinduced} yields a isomorphism of graded Poisson algebras. To do so in a natural way, we must work with the Kazhdan grading on $\k[\O_\chi]$. 

Recall that $\lambda_e:\k^\times\to G$ is an associated cocharacter for $\chi$. Equip $\k[G]$ with the grading coming from the $\k^\times$-action $\k^\times\times G\to G$, $(t,g)\mapsto \lambda_e(t)g\lambda_e(t)^{-1}$. The Poisson algebra $\k_\chi[\g^*]$ is equipped with the Kazhdan grading. Since the corresponding $\k^\times$-action on $\k_\chi[\g^*]$ induces a $\k^\times$-action on the set of its Poisson ideals, the unique maximal Poisson ideal $P_\chi$ is graded; hence the Kazhdan grading descends to $\k_\chi[\g^*]/P_\chi$. As in Subsection~\ref{ss:filteredgradedspaces}, $\k[G]\otimes \k_\chi[\g^*]/P_\chi$ becomes a graded algebra.

Also notice that $\k[\O_\chi] = \k[\cO_\chi]$ (Lemma~\ref{L:normalorbits}) is $\gamma_e(\k^\times)$-stable and so $\k[\O_\chi]$ inherits the Kazhdan grading from $\k[\g^*]$, by the remarks following \eqref{eq:Kazhdanaction}. Equivalently, we can equip $\k[\O_\chi]$ with the Kazhdan grading by twisting the PBW grading by $\lambda = \lambda_e$.

\begin{Lemma}\label{L:coordgrad}\label{C:coordisom}
\begin{enumerate}
\setlength{\itemsep}{4pt}
    \item     $\Ind_{G^{F(\chi)}}^G (\k_\chi[\g^*] / P_\chi)$ is a graded subalgebra of $\k[G]\otimes \k_\chi[\g^*]/P_\chi$.
    \item The isomorphism $\Ind_{G^{F(\chi)}}^G (\k_\chi[\g^*] / P_\chi) \isoto \k[\O_\chi]$ from Corollary~\ref{C:coordasinduced} is Kazhdan graded.
\end{enumerate}

\end{Lemma}
\begin{proof}
    Note first that $G^{F(\chi)}$ is normalised inside $G$ by $\lambda_e(\k^\times)$. Define a $G^{F(\chi)}$-action and a $\k^\times$-action on $X:=G\times \g^*$ by $h\cdot(g,\eta)=(gh^{-1},h\eta)$ and $t\cdot(g,\eta)=(\lambda_e(t)g\lambda_e(t)^{-1},t^{-2}\lambda_e(t)\eta)$, for $t\in \k^\times$ and $h\in G^{F(\chi)}$. 
    
    Write $\tau:G^{F(\chi)}\to\Aut(X)$ and $\sigma:\k^\times\to \Aut(X)$ for the corresponding maps. It is straightforward to check that $\sigma(\k^\times)$ normalises $\tau(G^{F(\chi)})$ inside $\Aut(X)$. This implies that the corresponding $\k^\times$-action on $\k[G]\otimes \k[\g^*]$ normalises the $G^{F(\chi)}$-action.

    The actions of $\k^\times$ and $G^{F(\chi)}$ on $\k[G]\otimes \k_\chi[\g^*]/P_\chi$ are inherited from $\tau$ and $\sigma$, thus we deduce the $\k^\times$-action on $\k[G]\otimes \k_\chi[\g^*]/P_\chi$  normalises the $G^{F(\chi)}$-action. This suffices to prove (1).

    Part (2) follows from the fact that the coadjoint action map $G\times \g^*\to \g^*$ is $\k^\times$-equivariant with respect to the $\k^\times$-actions $\tau$ and $\sigma$.
\end{proof}

\subsection{Constructing quantizations as induced modules}\label{ss:primidlstoquants}

This section contains one of the key constructions of the paper. Corollary~\ref{C:coordasinduced} tells us that we can construct the coordinate ring of $\O_\chi$ as a Poisson algebra by starting with a Poisson ideal of $\k[\g^*]$ of codimension $p^{d_\chi}$ which contains $J_\chi$. We demonstrate in this section that this procedure is amenable to quantization.

\begin{Lemma}\label{L:actiondescends}
    For $I \in \Prim_\eta U(\g)$, the action $G^{F(\eta)} \curvearrowright U(\g)$ descends to an action on $U(\g)/I$.
\end{Lemma}
\begin{proof}
First of all, note that $G^\eta \curvearrowright U_\eta(\g)$ for $\eta\in \g^*$, and therefore $G^\eta$ acts by conjugation on the set of primitive ideals $\Prim_\eta U(\g)$. Since $U_\eta(\g)$ is finite-dimensional the set is finite, and the connected component $(G^\eta)^\circ$ acts trivially. For $G = \GL_N$ all coadjoint centralisers are connected, and thus $G^\eta$ acts trivially.

    Recall from Section~\ref{ss:twistedcoadjoint} that $G^{F(\eta)}$ is generated by the subgroups $G_1$ and $G^\eta$. Therefore $G^{F(\eta)}$ acts on any space where $G_1$ and $G^\eta$ act compatibly (i.e. on any module for the Harish-Chandra pair $(G^\eta, \g)$). They do so on $U(\g)$ and hence on any quotient stabilised by both $\ad(\g)$ and $G^\eta$. By the remarks of the previous paragraph, this holds for $U(\g)/I$.
\end{proof}

Let $\eta\in\chi+X$ be an element in the Katsylo section. For $I\in\Prim_{\eta}^{p^{d_\chi}}U(\g)$ we define
\begin{eqnarray}
\label{eq:defineAI}
    \A_I := \Ind_{G^{F(\eta)}}^G (U(\g)/I).
\end{eqnarray} Note that Frobenius reciprocity \cite[Proposition I.3.4]{JanRAGS} gives a natural algebra homomorphism $\Phi_I:U(\g)\to \cA_I$.

For the next theorem, we equip $\k[\O_\chi]$ (and other algebras arising in the proof) with the Kazhdan grading so that we may apply Corollary~\ref{C:coordisom}. Recall that the Kazhdan grading on $\k[\O_\chi]$ may be obtained from the PBW grading via the twisting construction from Subsection~\ref{ss:filteredgradedspaces}. We omit twisting notation for simplicity.
\begin{Theorem}\label{T:grAI}
    $\A_I$ admits a filtration such that $\gr \A_I = \k[\O_\chi]$, where $\k[\O_\chi]$ is Kazhdan graded.
\end{Theorem}
\begin{proof}
    The cocharacter $\lambda_e:\k^\times\to G$ associated to $\chi$ defines a $\k^\times$-action on $G$ by conjugation and thus a grading $\k[G]=\bigoplus_{i\in\Z}\k[G]_i$. Equip $U(\g)$, and hence $U(\g)/I$, with the Kazhdan filtration. Then $\k[G]\otimes U(\g)/I$ is filtered as in Subsection~\ref{ss:filteredgradedspaces} and $$\gr(\k[G]\otimes U(\g)/I)\isoto \k[G]\otimes \gr(U(\g)/I),$$ since both tensorands admit filtered bases, see the remarks following \eqref{eq:gradedandtensors}. The inclusion $\cA_I\hookrightarrow \k[G]\otimes U(\g)/I$ induces the desired filtration.

    We show $\gr(\cA_I)=\k[\O_\chi]$ in 5 steps.

    \noindent{\bf Step 1:} $\gr(U(\g)/I)=\k_\chi[\g^*]/P_\chi$.

    Since $I\in\Prim_{\eta}^{p^{{d_\chi}}} U(\g)$, $U(\g)/I$ is a $p^{d_\chi}$-dimensional quotient of $U_\eta(\g)$. This step thus follows by combining Lemma~\ref{L:PremetSkryabin} and Lemma~\ref{L:grUchi}. 

    \noindent{\bf Step 2:} $\gr(\k[G^{F(\eta)}])=\k[G^{F(\chi)}]$.

    Using \eqref{eq:coordfrobstabiliserfibreproduct}, this follows from Lemma~\ref{L:coneisGchi} so long as $\k[G]$ is filtered free over $\k[G]^p$ (see \cite[Lemma I.8.2]{NVO}). Clearly $\k[\g]$ is filtered free over $\k[\g]^p$; that $\k[\g]_\delta$ is filtered free over $(\k[\g]_\delta)^p$ (where $\delta\in\k[\g]$ is the determinant function) follows easily from the identifications $(\k[\g]_\delta)^p=\k[\g]^p_{\delta^p}$ and $\k[\g]_\delta=\k[\g]\otimes_{\k[\g]^p}\k[\g]^p_{\delta^p}$.

    \noindent{\bf Step 3:} If $H$ and $K$ are subgroups of $G$ such that $K$ is preserved by the $\k^\times$-action and $\gr(\k[H])=\k[K]$, then $\gr\Ind_H^G(M)\subseteq \Ind_K^G\gr(M)$ for any finite-dimensional filtered $M\in H\lmod$.

    This follows by considering the associated graded homomorphism of the $\k[H]$-comodule structure map for $\k[G]\otimes M$. Note that the finite-dimensionality of $M$ ensures that $\k[G]\otimes M$ has a filtered basis, and thus \eqref{eq:gradedandtensors} gives an isomorphism.
    
   \noindent {\bf Step 4:} Combining Steps 1--3 yields $\gr(\cA_I)\subseteq \k[\O_\chi]$. Furthermore, the composition $\k[\g^*]\to\gr(\cA_I)\hookrightarrow \k[\O_\chi]$ is the comoment map for $\O_\chi$.

    The first of these observations is immediate. The second follows straightforwardly from Corollary~\ref{C:coordasinduced} and Frobenius reciprocity.

   \noindent {\bf Step 5:} $\gr(\cA_I)=\k[\O_\chi]$.

    This follows from Step 4, recalling that (for $G=\GL_N$) the comoment map $\k[\g^*]\to \k[\O_\chi]=\k[\o{\O}_\chi]$ is surjective.
\end{proof}

Combining the elements of the proof of Theorem~\ref{T:grAI}, the composition $$\gr(\cA_I)=\gr(\Ind_{G^{F(\eta)}}^G (U(\g)/I))\hookrightarrow \Ind_{G^{F(\chi)}}^G(\gr(U(\g)/I))\isoto \Ind_{G^{F(\chi)}}^G(\k_\chi[\g^*]/P_\chi)\isoto \k[\O_\chi]$$ is an isomorphism of algebras, which we denote $\iota_I$. Using the results of Subsection~\ref{ss:coordringdesc}, it is straightforward to check that $\iota_I$ is a Kazhdan graded Poisson isomorphism.

Furthermore, the proof of Theorem~\ref{T:grAI} shows that the natural $G$-action on $\A_I$ and homomorphism $\Phi_I:U(\g)\to \A_I$ satisfy (Q1) and (Q2) from Subsection~\ref{ss:gradedHamQuant}.

We collect these observations in the following corollary. Since we mainly consider $\k[\O_\chi]$ with its PBW grading, we will use the twisting notation $\k[\O_\chi]^{\lambda_e}$ for the rest of Section~\ref{s:Quants}.

\begin{Corollary}\label{C:AIQuant}
$(\cA_I,\iota_I,\Phi_I)\in\Quant^G(\k[\O_\chi]^{\lambda_e})$.
\end{Corollary}

Varying $\eta$ over $\chi+X$, this yields a map $$\Prim_{\chi+X}^{p^{d_\chi}} U(\g)\to \Quant^G(\k[\O_\chi]^{\lambda_e}) .$$ Recall from Lemma~\ref{L:twistquant} the bijection $\Quant^G(\k[\O_\chi]^{\lambda_e})\isoto \Quant^G(\k[\O_\chi])$, and from Proposition~\ref{P:quantizationsvsideals} the bijection $$\Quant^G(\k[\O_\chi])\to \Idl_{\o{\O}_\chi}^G U(\g), \qquad (\cA,\iota,\Phi)\mapsto \ker(\Phi).$$ Composing these maps yields $$\Omega:\Prim_{\chi+X}^{p^{d_\chi}} U(\g)\to \Idl_{\o{\O}_\chi}^G U(\g).$$

Finally, we note the following for future reference.
\begin{Lemma}\label{L:suppAI}
    Let $I\in \Prim_{\chi+X}^{p^{d_\chi}} U(\g)$ and $\eta\in\chi+X$. If $I_\eta\subseteq I$, then $\Supp_p(\A_I)=\o{\O}_\eta^{(1)}$. 
\end{Lemma}
\begin{proof} Using the fact that $\Omega(I)$ is $G$-stable, along with Lemma~\ref{L:whatissupport}, it suffices to show $\Omega(I)+I_\eta\neq U(\g)$. Denote by $\varepsilon$ the counit of $\k[G]$ and consider the composition $$U(\g)\to (\k[G]\otimes U(\g)/I)^{G^{F(\eta)}}\xrightarrow{\varepsilon\otimes 1} U(\g)/I,$$ which coincides with the natural surjection $U(\g)\twoheadrightarrow U(\g)/I$ by the axioms of a $\k[G]$-comodule. Clearly $\Omega(I)$ and $I_\eta$ are in the kernel of this non-zero map, which proves the result.
\end{proof}

\subsection{Classification of quantizations}\label{ss:classquants}

Recall that $G=\GL_N$, let $\blambda\vdash N$ be the partition corresponding to $\chi$, and let $\g_0$ be the Levi subalgebra of $\g$ introduced in Subsection~\ref{ss:subalgebrasassociated}.

Set $\z^*=\kappa\z(\g_0)$, let $W(\g_0)$ be the relative Weyl group of $\g_0$, and recall from Subsection~\ref{ss:subalgebrasassociated} that $W(\g_0)$ acts on $\z^*$ via the dot-action. Throughout this subsection, $\k[\O_\chi]$ is equipped with the PBW grading.

Thanks to \eqref{eq:1dimsandprimitives2} and Theorem~\ref{T:OnedimensionalsviaMiura}, we have natural bijections
$$\z^*/W(\g_0)_\bullet \overset{1\text{-}1}{\to} \E(\g,e) \overset{1\text{-}1}{\to} \Prim_{\chi+X}^{p^{d_\chi}} U(\g).$$

We claim that the set of isomorphism classes of quantizations, $\Quant^G \k[\cO_\chi]$, is naturally in bijection with $\Prim_{\chi + X}^{{p^{d_\chi}}} U(\g)$, and this claim will complete the proof of the first part of our main result (Section~\ref{ss:statementofresults}). 

In Proposition~\ref{P:quantizationsvsideals} we showed that taking the kernel of a quantum comoment map gives a bijection $\Quant^G \k[\o\O_\chi] \overset{1\text{-}1}{\to} \Idl_{\cO}^G U(\g)$, and so it remains to check that there is a bijection
\begin{eqnarray}
\label{eq:lastbigectiontodo}
    \Idl_{\cO_\chi}^G U(\g) \overset{1\text{-}1}{\to} \Prim_{\chi+X}^{p^{d_\chi}} U(\g).
\end{eqnarray}

Recall the constructions of the maps $$\Theta:\Idl_{\o{\O}_\chi}^{G} U(\g)\to \Prim_{\chi+X}^{p^{d_\chi}} U(\g) \qquad\mbox{and}\qquad \Omega:\Prim_{\chi+X}^{p^{d_\chi}} U(\g)\to \Idl_{\o{\O}_\chi}^{G} U(\g),$$ from Subsections~\ref{ss:quantisationtoprimitiveideals} and \ref{ss:primidlstoquants}, respectively. We prove that they are mutually inverse.

\begin{Proposition}\label{P:idltoidl}
    $\Omega\circ\Theta=\Id$.
\end{Proposition}
\begin{proof}
    Let $I\in \Idl_{\o{\O}_\chi}^{G} U(\g)$, and suppose that $\Theta(I)\supseteq I_\eta$ for $\eta\in \chi+X$. It is clear from the construction that $I\subseteq \Theta(I)$. Furthermore, since $I$ is $G$-stable we have $I\to \k[G]\otimes\Theta(I)$ under the $\k[G]$-comodule map, and thus that $I$ is in the kernel of $$U(\g)\to \Ind_{G^{F(\eta)}}^{G}(U(\g)/\Theta(I))=\cA_{\Theta(I)}.$$ By construction, this means $I\subseteq \Omega\Theta(I)$ and thus that there exists a filtered algebra surjection $$U(\g)/I\twoheadrightarrow U(\g)/\Omega\Theta(I).$$ The associated graded morphism is an isomorphism $\k[\O_\chi]\isoto\k[\O_\chi]$, which implies that $U(\g)/I\twoheadrightarrow U(\g)/\Omega\Theta(I)$ must in fact be bijective. Hence $\Omega\Theta(I)=I$ as claimed.
\end{proof}

\begin{Proposition}\label{P:primtoprim}
    $\Theta\circ\Omega=\Id$.
\end{Proposition}

\begin{proof}
    Let $I\in\Prim_{\chi+X}^{p^{d_\chi}} U(\g)$ with $I_\eta\subseteq I$. Set $\cA=U(\g)/\Omega(I)$ and let $\Phi:U(\g)\twoheadrightarrow \cA$ be the natural map. We first prove the following claim.
    
    {\bf Claim:} $\dim(\cA/\Phi(I_\eta)\cA)={p^{d_\chi}}$.

    {\em Proof of Claim:} Equip $U(\g)$ (and all subspaces and quotients thereof) with the Kazhdan filtration. By Lemma~\ref{L:grUchi}, $\gr(I_\eta)=J_\chi$ and thus we have inclusions $\mu^*(J_\chi)\k[\O_\chi]\subseteq \gr(\Phi(I_\eta))\gr(\cA)\subseteq \gr(\Phi(I_\eta)\cA)$ and hence a surjective Poisson homomorphism $$\k_\chi[\g^*]/P_\chi\isoto \k[\O_\chi]/\mu^*(J_\chi)\k[\O_\chi]\twoheadrightarrow \gr(\cA/\Phi(I_\eta)\cA),$$ where the first isomorphism follows from Lemma~\ref{L:inducedsimplePoissonquotient}. Since $\cA/\Phi(I_\eta)\cA\neq 0$ by Lemma~\ref{L:suppAI} and $\k_\chi[\g^*]/P_\chi$ is Poisson simple of dimension ${p^{d_\chi}}$ by Lemma~\ref{L:PremetSkryabin}, the claim follows.

    Note now that $\Omega(I)\subseteq I$ by the axioms of the adjoint $\k[G]$-comodule structure on $U(\g)$, and thus $\Omega(I)+I_\eta\subseteq I$. Furthermore, $\Omega(I)\subseteq \Theta\Omega(I)$ and $I_\eta \subset \Theta\Omega(I)$ by the arguments of Subsection~\ref{ss:quantisationtoprimitiveideals}, so similarly $\Omega(I)+I_\eta\subseteq \Theta(\Omega(I))$. There are thus surjections $$U(\g)/\Theta\Omega(I) \twoheadleftarrow U(\g)/(\Omega(I)+ I_\eta) \twoheadrightarrow U(\g)/I.$$ By \eqref{eq:Premetstheorem}, the above claim implies $\Theta\Omega(I)=\Omega(I)+ I_\eta=I$.
\end{proof}

We have now given a bijection \eqref{eq:lastbigectiontodo}, and this proves the following.
\begin{Theorem}\label{T:MainThm}
    We have natural bijections
    \begin{eqnarray}
        \label{eq:allthebijections}
         \z^*/W(\g_0)_\bullet \overset{1\text{-}1}{\to}\E(\g,e)\overset{1\text{-}1}{\to} \Prim_{\chi+X}^{{p^{d_\chi}}} U(\g)\overset{1\text{-}1}{\to} \Idl_{\o\O_\chi}^G U(\g)\overset{1\text{-}1}{\to} \Quant^G \k[\O_\chi].
    \end{eqnarray}
\end{Theorem}

We now turn to part (2) of the Theorem from the Introduction. Recall that $U(\g)^G$ is a central subalgebra of $U(\g)$ which is isomorphic to $\k[\t^*]^{W_\bullet}$. Given $\lambda\in\t^*/W_\bullet$, denote by $I^\lambda$ the corresponding maximal ideal of $U(\g)^G$.

\begin{Theorem}
\label{C:quantshavecharacters}
    Let $I\in\Idl_{\cO_\chi}^G U(\g)$. There exists $\lambda\in\t^*/W_\bullet$ such that $I^\lambda\subseteq I$, so $U(\g) / I$ admits a central character. Furthermore, \eqref{eq:commdiagCC} commutes. 
\end{Theorem}

\begin{proof}
    By Theorem~\ref{T:MainThm}, $I=\Omega(J)$ for some unique $J\in\Prim_{\chi+X}^{p^{d_\chi}} U(\g)$. Since $J$ is primitive, there exists unique $\lambda\in \t^*/W_\bullet$ such that $I^\lambda\subseteq J$. The $G$-stability of $I^\lambda$ in $U(\g)$ then implies that it lies in the kernel of the map $U(\g)\to\k[G]\otimes U(\g)/J$ induced from the structure map for the adjoint $\k[G]$-comodule structure on $U(\g)$. This kernel is precisely $\Omega(J)=I$.

    In particular, by Proposition~\ref{P:quantizationsvsideals} there exists a well-defined map $\Quant^G\k[\O_\chi]\to\t^*/W_\bullet$. Combining this with \eqref{eq:allthebijections} and the natural map $\z^*/W(\g_0)_\bullet\to \t^*/W_\bullet$ yields \eqref{eq:commdiagCC}. What remains is to show that \eqref{eq:commdiagCC} commutes.
    
    To do so, it will be helpful to use an alternative description of the finite $W$-algebra $U(\g,e)$ when $\g=\gl_N$. Recall all the notation from Section~\ref{S:generallinearalgebra}, and set $\r^{-}=\sspan\{e_{i,j} \mid \col(i) < \col(j)\}$ and $\r^-_\chi=\{x-\chi(x)\mid x\in\r\}$. Define $Q_\chi:=U(\g)/U(\g)\r^-_\chi$; then $U(\g,e)\cong Q_\chi^{R^{-}}$ where $R^{-}$ is the unipotent algebraic group with Lie algebra $\r^-$. With this definition, $Q_\chi$ has a $(U(\g),U(\g,e))$-bimodule structure induced from the left and right multiplication in $U(\g)$. Furthermore,  the natural map $U(\g)^G\to Q_\chi$ has (central) image in $U(\g,e)$ and thus $U(\g)^G$ acts on each simple $U(\g,e)$-module via a character.
    
    At this point, what remains for the commutativity of \eqref{eq:commdiagCC} is to show that the following diagram commutes
\begin{eqnarray}
\label{eq:centchardiag1}
\begin{tikzcd}
\z^*\arrow[rr]   \arrow[d] & & \E(\g,e) \arrow[rr] \arrow[rrd]& &  \Prim_{\chi+X}^{p^{d_\chi}} U(\g)  \arrow[d]\\
\t^*/W_\bullet \arrow[rrrr, hook, two heads] & & &   & \g^*/\!/G
\end{tikzcd}
\end{eqnarray}
Note that the map $\E(\g,e)\to \Prim_{\chi+X}^{p^{d_\chi}} U(\g)$ sends $V$ to $\Ann_{U(\g)} (Q_\chi \otimes_{U(\g,e)}V)$. Given $u\in U(\g)^G$, $q\in Q_\chi$ and $v\in V$ we have $$u(q\otimes v)=uq\otimes v=qu\otimes v = q\otimes uv,$$ which suffices to show the commutativity of the upper-right triangle of \eqref{eq:centchardiag1}.

We now turn to the bottom left trapezium in \eqref{eq:centchardiag1}, which corresponds to the following (left-hand) diagram of commutative algebras, in which the bottom horizontal arrow is the Harish-Chandra homomorphism. The right-hand diagram is the abelianisation of the left, after composing $U(\g)^G\to U(\t)^{W_\bullet}\to U(\g_0/[\g_0,\g_0])$.
\begin{eqnarray}
\label{eq:centchardiag2}
\begin{tikzcd}U(\g_0/[\g_0,\g_0])   & &  U(\g,e)^{\ab} \arrow[ll]
& &  U(\g_0)   & &  U(\g,e) \arrow[ll]
\\
U(\t)^{W_\bullet} \arrow[u]  &   & U(\g)^G \arrow[u] \arrow[ll, "\sim", swap]
& &  & \arrow[ul] U(\g)^G \arrow[ur] &
\end{tikzcd}
\end{eqnarray}

In the right-hand diagram \eqref{eq:centchardiag2}, the arrow $U(\g)^G\to U(\g_0)$ coincides with a slight generalisation of the (non-twisted) Harish-Chandra homomorphism, which can be described as follows. Let $T_0 \subseteq T$ be a torus such that $\g^{T_0} = \g_0$. Then we have $U(\g)^{T_0} = U(\g_0) \oplus K$ where $K = (\r U(\g)+U(\g)\r^{-})^{T_0}$. Then $K$ is an ideal in $U(\g_0)^{T_0}$ giving a homomorphism $U(\g)^G \subseteq U(\g)^{T_0} \to U(\g_0)$.

We enhance the right-hand diagram \eqref{eq:centchardiag2} as follows. \begin{eqnarray}
\label{eq:centchardiag4}
\begin{tikzcd}
U(\g_0)   & & & &   U(\p) \arrow[llll, two heads] \\
& & U(\g)=U(\p)\oplus U(\g)\r_\chi^{-}  \arrow[urr, two heads]  & &  \\
U(\g)^G \arrow[urr, hook] \arrow[uu] \arrow[rrrr] &  & & & U(\g,e) \arrow[uu, hook] 
\end{tikzcd}
\end{eqnarray}
The bottom-right triangle commutes by construction, so the commutativity of \eqref{eq:centchardiag2} follows if we show the commutativity of the top-left triangle in \eqref{eq:centchardiag4}. If we write $u = u_0 + u_1 \in U(\g)^G \subseteq U(\g)^{T_0} = U(\g_0) \oplus K$ then the map $U(\g) \to U(\p)$ sends $u_1$ to an element of $U(\p) \r$, which maps to zero under $U(\p) \to U(\g_0)$. Therefore $u$ is sent to $u_0$ by either trajectory through the top-left triangle.

This completes the proof.
\end{proof}

\begin{Remark}
\label{R:whyGLN}
    Throughout this section we have assumed that $G=\GL_N$. This restriction is used for five reasons: 
    (1) to apply our version of Katsylo's Theorem (Theorem~\ref{T:Katsylo}), 
    (2) to apply Lemma~\ref{L:coneisGchi} and extend it to Frobenius neighbourhoods, 
    (3) to guarantee that coadjoint centralisers $G^\eta$ are connected, and thus deduce Lemma~\ref{L:actiondescends}, 
    (4) to ensure that the orbit closure $\cO_\chi$ is normal, and 
    (5) to apply the results of \cite{GT}.
    (A sixth reason, (6) to ensure that each nilpotent orbit lies in a unique sheet, is actually unnecessary; for example, Proposition~\ref{P:psupport} still holds if we replace $\bbS_\chi$ with the union of sheets containing $\O_\chi$.)

    In particular, the first two of these reasons are only relevant when considering quantizations whose $p$-support is different from $\cO_\chi^{(1)}$, and the fifth is only relevant for identifying $\Prim_{\chi+X}^{p^{d_\chi}} U(\g)$ with $\z^*/W(\g_0)_\bullet$. 
    Therefore, if we denote by $\Quant^G_{\chi} \k[\cO_\chi]$ the subset of Hamiltonian quantizations with $p$-support $\cO_\chi^{(1)}$ and by $\big(\Prim_\chi^{{p^{d_\chi}}} U(\g)\big)^{G^{F(\chi)}}$ the subset of $G^{F(\chi)}$-stable ideals in $\Prim_\chi^{{p^{d_\chi}}} U(\g)$, the arguments in this section suffice to show a bijection $$\big(\Prim_\chi^{{p^{d_\chi}}} U(\g)\big)^{G^{F(\chi)}} \simeq \Quant^G_{\chi} \k[\cO_\chi]$$ whenever $\cO_\chi$ is normal. This is spelled out in slightly more detail in the next section. 
\end{Remark}

\section{The Joseph ideal in positive characteristics}\label{s:Joseph}

Write $\g_\C = \Lie G_\C$ for a complex simple Lie algebra. Recall that the minimal nilpotent $G_\C$-orbit $\O_{\min}$ is the unique smallest non-zero orbit in $\Nc(\g_\C^*)$. For $G_\C$ classical, it can be easily characterised by the following partitions: $\lambda = (1^{N-2}, 2)$ when $G_\C = \SL_N$; $\lambda = (1^{N-4}, 2^2)$ for $G_\C = \SO_N$; $\lambda = (1^{N-2}, 2)$ for $N$ even and $G_\C = \Sp_N$. For $G_\C$ exceptional, $\O_{\min}$ is always the nilpotent orbit with Bala-Carter label $A_1$ (consisting of long roots when the root system is not simply laced).

Joseph proved \cite{Jo} that if the root system of $\g_\C$  does not have type {\sf A} then:
\begin{itemize}
\setlength{\itemsep}{4pt}
    \item[(i)] there is a unique completely prime ideal $J_0 \lhd U(\g_\C)$ such that the vanishing locus of $\gr J_0 \lhd S(\g) = \C[\g^*]$ is $\o\O_{\min}$.
    \item[(ii)] $J_0$ is a maximal ideal.
\end{itemize}

Our goal is to prove an analogue of Joseph's result over fields of positive characteristic. For the rest of the paper we assume $\k = \bar \k$ has characteristic $p>0$, and let $G$ be a simply connected simple algebraic group over $\k$. We assume that $p$ is very good for $G$.

As with the case over $\C$, there is a minimal nilpotent $G$-orbit $\O_{\min} \subseteq \Nc(\g^*)$, which is determined by the same partition (resp. Bala-Carter label) as above. The closure $\o\O_{\min}$ is always normal \cite[\textsection 8.6]{JaNO}.

Now we prove a general criterion to show that certain orbits admit a unique quantization. Here, a nilpotent orbit is {\em  rigid} if it is equal to a sheet; this definition coincides with the usual one in terms of Lusztig-Spaltenstein induction (see e.g. \cite{PS}).
\begin{Theorem}
\label{L:uniquequant}
  Suppose that $\O \subseteq \g^*$ is a rigid nilpotent orbit not lying in Table~\ref{ta: BadOrbs}, such that $\cO$ is normal. There is a unique $G$-stable two-sided ideal $J_\O \lhd U(\g)$ such that $\gr J_\O$ is the defining ideal of $\cO$. It is completely prime.
\end{Theorem}
\begin{proof}
    Let $\chi \in \O$ and write $d_\chi = \dim G\cdot \chi$. Maintaining notation from Remark~\ref{R:whyGLN}, we show that there is a bijection
    \begin{eqnarray}
        \label{eq:ogsog}
        \left(\Prim_\chi^{p^{d_\chi}}U(\g)\right)^{G^{F(\chi)}} \to \Idl_{\cO}^G U(\g).
    \end{eqnarray}
    Let $I\in \Prim_\chi^{p^{d_\chi}} U(\g)$ be $G^{F(\chi)}$-stable and define $\A_I$ as per \eqref{eq:defineAI}. The proof of Theorem~\ref{T:grAI} shows that $\gr \A_I = \k[\o\O]$ with respect to the Kazhdan grading (note that Step 2 of that proof is trivial when $\eta=\chi$).  As in Lemma~\ref{L:twistingfiltrations} we may twist the filtration such that $\gr \A_I = \k[\o\O]$ with respect to the PBW grading. Since $\cO$ is normal, we have $\k[\g^*] \onto \k[\O] = \k[\o\O]$, and so $U(\g) \onto \A_I$ with kernel in $\Idl_{\cO}^G U(\g)$.

    Now suppose $(\cA,\iota,\Phi)\in\Quant^G(\k[\O])$. By Proposition~\ref{P:psupport} (see also Remark~\ref{R:whyGLN}) we know that $\Supp_p(\A)$ is contained in the union of closures of sheets containing $\O$. Since $\O$ is rigid, this union of closures equals $\o\O$, and by \cite[Lemma 5.10]{BDT} $\chi+X=\{\chi\}$. Proposition~\ref{P:fibredimension} then shows that  the kernel $K$ of $U(\g) \to \A/\Phi(I_\chi)\cA$ has codimension $p^{d_\chi}$ and contains $I_\chi$. Therefore, $K$ is the annihilator of a simple module of dimension $p^{\frac{1}{2}d_\chi}$, by Premet's theorem \eqref{eq:Premetstheorem}. It is clearly $G^{F(\chi)}$-stable.

    The two previous paragraphs show that we have a map \eqref{eq:ogsog} and a map in the opposite direction. They are inverse to each other, thanks to Propositions~\ref{P:idltoidl} and \ref{P:primtoprim}. 

    Under the stated assumptions, $\Prim_\chi^{p^{d_\chi}} U(\g)$ consists of a single ideal \cite[Remark 5.5]{PT20}; such ideal must be $G^{F(\chi)}$-stable.  We have thus shown existence and uniqueness of $J_\O$. That $J_\O$ is completely prime follows from the fact that $\gr \A_I$ is a domain.

\end{proof}

	\renewcommand{\arraystretch}{1.5}
	\begin{center}
    \begin{table}[h] \caption{Exceptions to Theorem~\ref{L:uniquequant}.}\label{ta: BadOrbs}
		\begin{tabular}{|c||c|c|c|c|c|c|}
        \hline
			Type of $G$ & $\sG$ & $\sF$ & $\sEj$ & $\sEk$ & $\sEk$ & $\sEk$ \\
			\hline
			Bala-Carter label & $\widetilde{\sA}_{\mathsf 1}$ & $\widetilde{\sA}_{\mathsf 2}+\sAa$ & $(\sAc+\sAa)'$ & $\sAc+\sAa$ & $\sAe+\sAa$ & $\sDe(\saa)+\sAb$ \\
            \hline
		\end{tabular}
        \end{table}
	\end{center}
\begin{Remark}
    Over $\C$, the six rigid nilpotent orbits in Table~\ref{ta: BadOrbs} do not have normal closure in $\g^*$ \cite{Br1,Br2,Kr}. We expect this to remain true in good characteristic, but it appears that this is not currently known.     
\end{Remark}

\begin{Corollary}\label{C:UniqMinQuant}
    If $G$ does not have type $\sA$ then there is a unique $G$-stable ideal $J_{\O}\lhd U(\g)$ such that $\gr J_{\O} = \k[\o\O_{\min}]$. It is completely prime, but not maximal amongst $G$-stable ideals. $\hfill \qed$
\end{Corollary}
    
\begin{Remark}
    It is interesting that the ideal $I$ in Theorem~\ref{L:uniquequant} is not maximal among $G$-stable ideals of $U(\g)$ (Cf. (ii) above). Let $\A := U(\g) / I$. Since $\O \subseteq \Supp_p(\A)$ and $\Supp_p(\A)$ is closed, it contains zero. Therefore, writing $I_0 = (x^p - x^{[p]} \mid x\in \g) \subseteq Z_p(\g)$ yields $\A_0 := \A/I_0\A \ne 0$. Thus we have $G$-stable ideals
    $$J_\O \subsetneq I_0U(\g) + J_\O \subsetneq U(\g) $$
\end{Remark}

\begin{Remark}
    Theorem~\ref{L:uniquequant} does not apply in type $\sA$ since $\O_{\min}$ is not rigid in that case. The set $\Idl_{\cO_{\min}}^G U(\g)$ is however completely understood in type $\sA$ due to Theorem~\ref{T:MainThm}; in particular, for $G=\SL_N$ with $p\nmid N$ and $N> 2$ it may be identified with $\mathbb{A}_\k^1$.
\end{Remark}

\end{document}